\documentclass[11pt]{article}
\usepackage[a4paper,margin=2.4cm]{geometry}
\usepackage[T1]{fontenc}
\usepackage{lmodern,microtype,amsmath,amssymb,amsthm,booktabs,longtable,graphicx,enumitem,array}
\usepackage{placeins,adjustbox,flafter,float}
\usepackage{xurl}
\usepackage[numbers,sort&compress]{natbib}

\usepackage[colorlinks=true,allcolors=blue,bookmarksdepth=2]{hyperref}
\hypersetup{
  pdftitle={Measuring Sudoku Difficulty: Formulation, Propagation and Uniqueness Certification},
  pdfauthor={Aled Williams}
}

\theoremstyle{plain}
\newtheorem{proposition}{Proposition}

\newtheorem{corollary}[proposition]{Corollary}
\newcommand{\C}{\mathcal C}
\newcommand{\Z}{\mathcal Z}
\newcommand{\ind}{\mathbf 1}
\newcommand{\mainref}[1]{\ref{#1}}
\newcommand{\NInstances}{32,000}

\newcommand{\NMargin}{320}

\newcommand{\CountGroupSize}{500}

\newcommand{\CaseFirstDigit}{8}
\newcommand{\CaseHardIterations}{16,367}
\newcommand{\RecoveredFindN}{32,000}
\newcommand{\RecoveredCertN}{32,000}
\newcommand{\SinglesCompleteN}{29,658}
\newcommand{\ExternalN}{105}
\newcommand{\RecoveredAllN}{32,000}
\newcommand{\RecoveredAllR}{0.991}
\newcommand{\RecoveredActiveN}{1,962}
\newcommand{\RecoveredActiveR}{0.846}
\newcommand{\RecoveredBothN}{1,928}
\newcommand{\RecoveredBothR}{0.841}
\newcommand{\NaturalAllFindBaseR}{0.653}
\newcommand{\NaturalAllFindBaseMAE}{0.416}
\newcommand{\NaturalAllFindFullR}{0.660}
\newcommand{\NaturalAllFindFullMAE}{0.481}
\newcommand{\NaturalAllCertBaseR}{0.654}
\newcommand{\NaturalAllCertBaseMAE}{0.436}
\newcommand{\NaturalAllCertFullR}{0.660}
\newcommand{\NaturalAllCertFullMAE}{0.502}
\newcommand{\NaturalPositiveFindBaseR}{0.266}

\newcommand{\NaturalPositiveFindFullR}{0.335}

\newcommand{\NaturalPositiveCertBaseR}{0.283}

\newcommand{\NaturalPositiveCertFullR}{0.357}

\newcommand{\CPAllFindFullR}{0.947}

\newcommand{\CPAllFindIndicatorR}{0.898}

\newcommand{\CPAllCertFullR}{0.924}

\newcommand{\CPAllCertIndicatorR}{0.891}

\newcommand{\CPResidualFindBaseR}{0.049}

\newcommand{\CPResidualFindFullR}{0.495}

\newcommand{\CPResidualCertBaseR}{0.029}

\newcommand{\CPResidualCertFullR}{0.335}

\newcommand{\InitialUnfinished}{2,576}
\newcommand{\PositiveFindN}{1,955}
\newcommand{\PositiveCertN}{1,935}
\newcommand{\SeventeenFindMedian}{43,134.5}

\newcommand{\SeventeenFindMaximum}{4,352,781}
\newcommand{\SeventeenFindZero}{41}
\newcommand{\SinglesPositiveFind}{1,103}
\newcommand{\NaturalLowFindBaseR}{0.352}
\newcommand{\NaturalLowFindFullR}{0.443}

\newcommand{\NaturalPositiveFindGain}{0.069}
\newcommand{\NaturalPositiveFindGainLow}{0.051}
\newcommand{\NaturalPositiveFindGainHigh}{0.087}
\newcommand{\PositiveModelFindN}{1,947}
\newcommand{\ExternalNaturalFindDone}{103}

\newcommand{\RanksFindR}{0.554}
\newcommand{\NestedFindUp}{38}
\newcommand{\NestedFindDone}{2,489}
\newcommand{\NestedFindNonTieFraction}{0.236}
\newcommand{\NestedFindBases}{17}
\newcommand{\NestedFindMedianUp}{3,860}

\newcommand{\SeventeenCertMedian}{88,149.0}

\newcommand{\SeventeenCertMaximum}{3,695,730}
\newcommand{\SeventeenCertZero}{40}
\newcommand{\SinglesPositiveCert}{1,096}
\newcommand{\NaturalLowCertBaseR}{0.358}
\newcommand{\NaturalLowCertFullR}{0.446}

\newcommand{\NaturalPositiveCertGain}{0.074}
\newcommand{\NaturalPositiveCertGainLow}{0.055}
\newcommand{\NaturalPositiveCertGainHigh}{0.095}
\newcommand{\PositiveModelCertN}{1,928}
\newcommand{\ExternalNaturalCertDone}{96}

\newcommand{\RanksCertR}{0.580}
\newcommand{\NestedCertUp}{27}
\newcommand{\NestedCertDone}{2,463}
\newcommand{\NestedCertNonTieFraction}{0.206}
\newcommand{\NestedCertBases}{14}
\newcommand{\NestedCertMedianUp}{4,664}

\newcommand{\RecoveredActiveLow}{0.829}
\newcommand{\RecoveredActiveHigh}{0.861}
\newcommand{\RecoveredBothLow}{0.826}
\newcommand{\RecoveredBothHigh}{0.857}
\newcommand{\BothZeroN}{30,038}
\newcommand{\FindOnlyN}{27}
\newcommand{\CertOnlyN}{7}
\newcommand{\PermutationCaseMin}{0}
\newcommand{\PermutationCaseMax}{15,563}

\newcommand{\PermutationFindDone}{247}
\newcommand{\PermutationFindVariableBases}{6}

\newcommand{\PermutationCertDone}{236}

\newcommand{\PermutationCertCompleteBases}{10}
\newcommand{\PermutationCertWithinVariance}{5.84}
\newcommand{\PermutationCertBetweenVariance}{13.04}
\newcommand{\FeatureBatchMin}{0.098}
\newcommand{\FeatureBatchMax}{0.105}
\newcommand{\SinglesBatchMin}{0.076}
\newcommand{\SinglesBatchMax}{0.077}
\newcommand{\CPFindBatchMin}{0.102}
\newcommand{\CPFindBatchMax}{0.107}
\newcommand{\CPCertBatchMin}{0.085}
\newcommand{\CPCertBatchMax}{0.086}
\newcommand{\AugmentationTwentyFourSolved}{37}
\newcommand{\DeletionTwentyFourSolved}{254}
\newcommand{\TwentyThreeSolved}{462}
\newcommand{\ChainBeforeLP}{5,690}
\newcommand{\ChainBeforeRows}{1,003}
\newcommand{\ChainBeforeCols}{533}
\newcommand{\ChainAfterLP}{55,415}
\newcommand{\ChainAfterRows}{986}
\newcommand{\ChainAfterCols}{518}
\newcommand{\ConditionalPositiveLPR}{0.781}
\newcommand{\ConditionalPositiveLPLow}{0.751}
\newcommand{\ConditionalPositiveLPHigh}{0.800}

\newcommand{\ConditionalPositiveNodesR}{0.748}

\newcommand{\ConditionalPositiveNodesN}{1,807}
\newcommand{\SerialTimingN}{88}

\newcommand{\CertCheaperN}{389}
\newcommand{\CertDearerN}{1,573}
\newcommand{\TaskDeltaMedian}{0.490}
\newcommand{\TaskDeltaQOne}{0.095}
\newcommand{\TaskDeltaQThree}{0.975}
\newcommand{\PairedPositiveRatio}{1.64}
\newcommand{\ReviewFindNegativeN}{15,501}

\newcommand{\ReviewFindClippedMAE}{0.445}

\newcommand{\ReviewCertNegativeN}{15,614}

\newcommand{\ReviewCertClippedMAE}{0.466}

\newcommand{\ValidationBaseOverlap}{0}
\newcommand{\MarginCrossOne}{44}
\newcommand{\MarginCrossTwo}{11}
\newcommand{\MarginCrossThree}{6}
\newcommand{\MarginCrossFour}{3}

\newcommand{\IndependentMarginN}{64}

\newcommand{\HardAssignmentFindPositive}{58}

\newcommand{\HardCrossFindR}{0.084}
\newcommand{\HardCrossCertR}{0.409}
\newcommand{\HardN}{64}
\newcommand{\HardMinClues}{21}
\newcommand{\HardMaxClues}{26}
\newcommand{\HardSourceN}{48,766}

\newcommand{\SubmissionNaturalAssignmentOrdinaryDifference}{0.373}
\newcommand{\SubmissionNaturalAssignmentOrdinaryLow}{0.056}
\newcommand{\SubmissionNaturalAssignmentOrdinaryHigh}{0.692}
\newcommand{\SubmissionNaturalAssignmentWithinDifference}{0.371}
\newcommand{\SubmissionNaturalAssignmentWithinLow}{-0.069}
\newcommand{\SubmissionNaturalAssignmentWithinHigh}{0.798}
\newcommand{\SubmissionRanksFindLess}{50}

\newcommand{\SubmissionRanksFindRatio}{0.630}
\newcommand{\SubmissionRanksFindR}{0.360}
\newcommand{\SubmissionRanksCertLess}{59}

\newcommand{\SubmissionRanksCertRatio}{0.611}
\newcommand{\SubmissionRanksCertR}{0.677}
\newcommand{\SubmissionCounterAssignmentR}{0.048}

\newcommand{\SubmissionCounterNaturalR}{0.394}

\newcommand{\SubmissionCounterCPR}{0.496}
\newcommand{\SubmissionCounterCPWithin}{0.441}

\newcommand{\SubmissionAssignmentPositiveN}{58}
\newcommand{\SubmissionAssignmentPositiveR}{0.100}
\newcommand{\SubmissionAssignmentPositiveLow}{-0.188}
\newcommand{\SubmissionAssignmentPositiveHigh}{0.351}
\newcommand{\SubmissionMarginPositiveN}{19}
\newcommand{\PreBenchmarkN}{320}
\newcommand{\PreExternalN}{64}
\newcommand{\PreAllN}{384}
\newcommand{\PrePrimaryCalls}{3,072}
\newcommand{\PrePrimaryFinished}{3,062}
\newcommand{\PrePrimaryUnresolved}{10}
\newcommand{\PreExtendedN}{9}
\newcommand{\PreExtendedCalls}{72}
\newcommand{\PreExtendedFinished}{72}

\newcommand{\PreRepeatN}{62}

\newcommand{\PreTimeLimit}{10}
\newcommand{\PreExtendedLimit}{60}

\newcommand{\PreAssignmentFindOffZero}{313}

\newcommand{\PreAssignmentCertOffZero}{311}
\newcommand{\PreHardAssignmentFindRank}{0.086}
\newcommand{\PreNaturalFindOnZero}{301}
\newcommand{\PreNaturalFindOffZero}{238}

\newcommand{\PreHardNaturalFindRank}{0.096}
\newcommand{\PreCommonOnTaskR}{0.798}
\newcommand{\PreCommonPositiveN}{19}
\newcommand{\PreHardAssignmentOnR}{0.031}

\newcommand{\PreHardNaturalOnR}{0.404}

\newcommand{\PreCommonOffTaskR}{0.833}
\newcommand{\PreHardAssignmentOffR}{0.075}

\newcommand{\PreHardNaturalOffR}{0.218}
\newcommand{\PreHardNaturalOffLow}{-0.066}
\newcommand{\PreHardNaturalOffHigh}{0.463}
\newcommand{\PreFigureCases}{2}
\newcommand{\PreFigureCalls}{8}
\newcommand{\PreFigureFinished}{8}

\newcommand{\PreFigureLeftCertOff}{16,681}

\newcommand{\PreFigureRightCertOff}{26,353}
\title{\LARGE{Measuring Sudoku Difficulty: Formulation, Propagation and Uniqueness Certification}}

\author{
    {\Large Aled Williams \vspace{3.0mm}}\\
    Department of Computer Science\\
    Royal Holloway, University of London\\
    London, UK \vspace{2.0mm} \\
    \texttt{aled.williams@rhul.ac.uk}
}
\date{}
\begin{document}
\maketitle
\begin{abstract}
Finding a Sudoku completion and proving that the completion is unique are different computational tasks. Despite both concerning the same puzzle, their computational effort need not be closely related. In this paper, we investigate how this relationship depends on the mathematical formulation, solver and measure of computational effort. Using \NInstances\ uniquely solvable puzzles and \HardN\ harder instances, we compare integer programming formulations and CP-SAT. Further, we prove that the raw natural formulation has an extremely weak linear programming relaxation, even with uniqueness exclusions. Computational experiments show that the relationship between completion and certification effort depends on the formulation, puzzle population, solver and effort measure, rather than on the puzzle alone.
\vspace{3.0mm}

\end{abstract}


\section{Introduction}
Finding a completed Sudoku grid and establishing that it is the only completion are different computational tasks. The first requires a grid satisfying the clues and Sudoku rules. The second requires ruling out every alternative. Note that to establish uniqueness, we start with a valid completion supplied in advance and run the solver afresh. Thus, the recorded effort excludes the cost of finding that reference completion.

We investigate how the relationship between these tasks depends on both the mathematical formulation and the solver operations being measured. Building on the assignment and natural cell-value models introduced in our earlier paper \cite{williams2026}, which focused on modelling choices motivated by pedagogy, we now examine the computational effort required to find a completion and establish its uniqueness. 

The following example motivates the comparison. Two seventeen-clue puzzles from our previously published Sudoku benchmark \cite{dataset2026} can both be completed using only naked and hidden singles. A \textit{naked single} identifies the only possible digit for a cell (e.g. if a row already contains digits one through eight, its only empty cell must contain nine). A \textit{hidden single} identifies the only possible cell for a digit within a row, column or block. For example, suppose the first two columns of a block contain digits one through six. The final column must therefore contain the digits seven, eight and nine. If other givens rule out seven in the first two cells of that column, the remaining cell must contain seven, even if other digits are still candidates there. The key distinction is that, in a naked single, the cell determines the digit, whereas in hidden single, the digit determines the cell. 

Yet, with the (raw) natural cell-value formulation (outlined in Section 4.2), the uniqueness-certification runs report zero and \CaseHardIterations\ linear programming (LP) iterations for the two puzzles in Figure~\ref{fig:opening}, respectively. It should be noted that in the zero-iteration case, the solver establishes uniqueness during presolve, when it simplifies the model before the main search. This contrast raises a modelling question: which deductions does a representation expose directly, and which must the solver recover?

With the HiGHS presolve option set to off, the same uniqueness tests require \PreFigureLeftCertOff\ and \PreFigureRightCertOff\ LP iterations, respectively. Section~\ref{sec:presolve} extends this check to a fixed sample of benchmark and harder puzzles; Appendix~\ref{app:presolve} records the protocol and the two opening cases.

\begin{figure}[ht!]\centering\includegraphics[width=\linewidth]{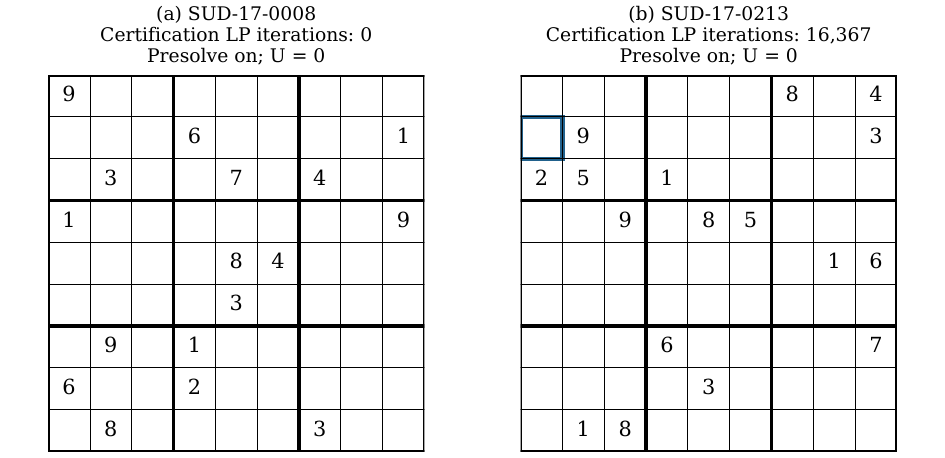}
\caption{This figure shows two seventeen-clue puzzles that can both be completed using only naked and hidden singles. Here $U$ denotes the number of cells unresolved after repeated singles. Note that the outlined cell within the right-hand puzzle is the only possible position for digit \CaseFirstDigit\ within its block. The reported certification LP iterations use the raw natural cell-value formulation with presolve enabled. Table~\ref{tab:opening-presolve} compares both tasks with presolve on and off.}
\label{fig:opening}\end{figure}

The above examples show that puzzles with the same number of clues can require very different computational effort. More generally, effort also depends on the formulation, solver, task (e.g. completion or certifying uniqueness), and operations measured. To make this dependence explicit, we write
$$
 D_{A,F,T}(P)
$$
for the measured computational effort (or difficulty) on puzzle $P$, using formulation $F$, algorithm (or solver) $A$, and task $T$. This is an instance of the broader question of how problem features relate to algorithmic performance \citep{smithmiles2012}. It should be emphasised that the measurements here concern computation, not necessarily human solving difficulty.

We compare binary assignment and explicit natural cell-value formulations using the same mixed-integer linear programming (MILP) solver, and study an all-different model implemented in CP-SAT separately. Every puzzle in the exact-count benchmark of \cite{dataset2026} has exactly one completion, meaning that reducing the clue count in this collection does not introduce puzzles with many admissible completions.

Three natural questions guide the paper:
\begin{itemize}
    \item How much variation remains among puzzles with the same clue count, and how does it depend on the formulation? 
    \item Are completion and fresh uniqueness certification difficult on the same puzzles? 
    \item Which structural features explain their recorded effort? 
\end{itemize}
Clue-deletion margins and matched clue additions provide focused extensions of the third question.

We address these questions through paired task measurements, comparisons of alternative formulations, and controlled changes to digit labels and clues. The mathematical analysis explains what the raw relaxations retain; the experiments show how strengthening and preprocessing affect recorded effort. Here, ``raw'' denotes the pairwise cell-value formulation before the additional equalities or separate singles preprocessing considered later, which does not mean that solver presolve is disabled. We compare it with strengthened variants and a high-level all-different model.

Appendix~\ref{app:math} gives additional mathematical details. Appendix~\ref{app:implementation} specifies the computational methods and validation, and Appendix~\ref{app:comparisons} contains the supporting experimental comparisons.


\section{Related work and notions of Sudoku difficulty}
\label{sec:related}

\subsection{Human ratings and propagation}
Pel{\'a}nek \citep{pelanek2014} evaluates difficulty metrics against human solving data, distinguishing the complexity of individual logical steps from dependencies between them. More recent approaches use SAT clause distributions and solving strategies \citep{eisenkolb2025}, clue and candidate features \citep{higgins2025}, and structural CSP metrics \citep{thangamaniCSP2026}. Notice et al. \citep{notice2023} use related features to analyse instances and predict metaheuristic solver performance, showing that conclusions depend on the chosen performance measure. These studies motivate structural measures of difficulty, although our focus is on solver operations rather than human ratings.

Simonis \citep{simonis2005} studies Sudoku through constraint propagation and model strengthening. Lynce and Ouaknine \citep{lynce2006} demonstrate the benefits of redundant clauses in SAT encodings, while redundancy among Sudoku constraints is examined by \cite{demoen2014}. These studies illustrate how equivalent formulations can lead to different solver behaviour. Provan \citep[Theorem~2]{provan2009} shows that completion by pigeon-hole rules implies uniqueness in the nonnegative assignment linear system, providing context for our propagation analysis. Dukes and Nimegeers \citep{dukes2024} further study the assignment linear system and fractional completions for generalised Sudoku.

\subsection{Clue density, structure and formulation}
Ans{\'o}tegui et al. \citep{ansotegui2006} investigate clue patterns and computational hardness in generalised Sudoku. Related studies examine simulated annealing and clue density \citep{lewis2007}, analogue SAT solver dynamics \citep{ercsey2012}, and density, backtracking and freezing transitions \citep{prates2018}. These results depend on the puzzle population being studied. In particular, random clue removal need not preserve uniqueness, whereas all puzzles in our experiments have unique completions.

Integer programming formulations appear in the teaching-oriented work of Chlond \citep{chlond2005} and Weiss and Rasmussen \citep{weiss2007}, with broader computational comparisons given by \cite{coelho2014}. Our earlier paper \cite{williams2026} introduced the assignment and natural cell-value formulations from a pedagogical perspective, while here we compare their computational behaviour. Thangamani and Regulagedda \citep{thangamani2026} compare CP, SAT and IP on uniqueness-validated puzzles, but do not measure completion and certification separately.

Lagerkvist \citep{lagerkvist2026} develops unique-solution corpora across grid sizes, studies propagation and clue addition, and compares alternative approaches to excluding known solutions. Our experiments similarly examine propagation and uniqueness, but record completion and fresh-reference exclusion separately for two integer formulations within the same solver. We also consider alternative exclusion constraints, digit relabellings and nested clue additions.

Dillon's exposition of Tdoku \citep{tdokuStory} illustrates how benchmark selection and representation affect solver performance. Its protocol searches for up to two solutions, whereas our certification task starts with a supplied completion and a fresh solver state.

\subsection{Uniqueness, defining sets and correction sets}
The Another Solution Problem asks whether a second solution exists when one solution is supplied. Yato and Seta \citep{yato2003} establish ASP-completeness for generalised Number Place. This asymptotic result does not directly describe computational effort for fixed $9\times9$ puzzles. Egidy et al. \citep{egidy2025} provide a broader analysis of second-solution complexity.

McGuire et al. \citep{mcguire2014} establish the minimum-clue result using unavoidable sets and hitting-set enumeration. A minimal puzzle loses uniqueness after any clue deletion, whereas a minimum-clue puzzle has the smallest possible number of clues. Critical determining sets provide a related framework \citep{cooper2014}. Tjusila et al. \citep{tjusila2024} formulate minimum clue selection as a bilevel problem, while Eppstein and Zhang \citep{eppstein2026} study how the required clue count depends on the completed grid.

Our clue-deletion measure is also related to correction sets: Sudoku rules and the exclusion constraint are fixed, while clues can be removed. The distinction between minimum cardinality and inclusion minimality is standard in this literature \cite{liffiton2008}. We therefore use the term \emph{minimum clue-deletion margin}.


\section{Benchmark dataset}\label{sec:data}
The benchmark \cite{dataset2026} contains \NInstances\ standard Sudoku puzzles that are uniquely completable, with \CountGroupSize\ instances at every exact clue count from $17$ through $80$. The release stores each cell's position, given digit and solution digit, while zero denotes a blank.

The two generation procedures need to be distinguished carefully when interpreting clue-count curves. For groups with up to $23$ clues, the generator uses distinct records from a pinned Royle collection, applies Sudoku-preserving symmetries and adds correct clues where needed. From $24$ clues onwards, the generator starts from randomised completed grids and removes clues while checking uniqueness. Neither procedure samples uniformly from all uniquely solvable puzzles. Thus, a change in observed effort at the boundary could reflect the generation procedure as well as clue count.

We reconstruct every puzzle and stored solution, check the clues and Sudoku rules, and inspect every archived uniqueness record. Independent integer enumeration provides a second uniqueness check on a fixed stratified sample. Appendix~\ref{app:data} gives the coverage, source-file audit and exact symmetry checks. It also distinguishes construction from the later AMPL/HiGHS verification, whose archived code provides no evidence of replacing puzzles after a verification timeout. 


\section{Completion and uniqueness-certification models}\label{sec:models}
Let $N=\{1,\ldots,9\}$, let $\C(P)\subseteq N\times N$ be the clue positions, and let $g_{ij}$ be the given digit at a clue position. Write $C=|\C(P)|$ for the number of clues, where $17 \le C \le 80$ applies to our benchmark. The set of all completed Sudoku grids is denoted by $\Z$. The reference completion $\bar z\in\Z$ satisfies every clue.

\subsection{Binary assignment formulation}
The assignment model records the choice of a digit for each cell. For each $i,j,k\in N$, let $x_{ijk}=1$ precisely when digit $k$ occupies cell $(i,j)$, and 0 otherwise. Each Sudoku rule can then be expressed by requiring a suitable group of these choices to sum to one. The formulation is
\begin{align}
 \sum_{i\in N}x_{ijk}&=1 &&j,k\in N,\label{eq:columns}\\
 \sum_{j\in N}x_{ijk}&=1 &&i,k\in N,\\
 \sum_{q=0}^{2}\sum_{r=0}^{2}x_{a+q,b+r,k}&=1 &&a,b\in\{1,4,7\},\ k\in N,\label{eq:blocks}\\
 \sum_{k\in N}x_{ijk}&=1 &&i,j\in N,\label{eq:cells}\\
 x_{ij,g_{ij}}&=1 &&(i,j)\in\C(P),\label{eq:clues}\\
 x_{ijk}&\in\{0,1\} &&i,j,k\in N.
\end{align}
It should be noted that constraints~\eqref{eq:columns}--\eqref{eq:blocks} require each digit to appear once in every column, row and block. Constraint~\eqref{eq:cells} assigns exactly one digit to each cell, and \eqref{eq:clues} fixes the given clues. Thus, every Sudoku rule becomes a linear equality. The model features $729$ binary variables and $324$ Sudoku equalities before clue fixing. Observe that completion is a feasibility problem, so we can use a constant objective. Further, the corresponding cell value is $z_{ij}=\sum_{k\in N}k x_{ijk}$.

Given a solution (completion), to test uniqueness, we require a candidate grid to disagree with the supplied completion in at least one cell. Observe that in the assignment variables this can be achieved by adding
\begin{equation}
 \sum_{i\in N}\sum_{j\in N}x_{ij,\bar z_{ij}}\le80,\label{eq:nogood}
\end{equation}
where $x_{ij,\bar z_{ij}}$ equals one if the candidate grid agrees with the supplied completion at cell $(i,j)$, and zero otherwise. Thus, the left-hand side counts the number of agreeing cells, and the inequality ensures that at least one cell differs. 

\begin{proposition}\label{prop:exclude}
The constraint \eqref{eq:nogood} excludes exactly the completed grid $\bar z$ from the assignment model. With the original clues imposed, the resulting feasibility problem is infeasible if and only if $\bar z$ is the unique completion.
\end{proposition}
\begin{proof}
By \eqref{eq:cells}, it follows that $x_{ij,\bar z_{ij}}$ is one exactly when the candidate grid agrees with $\bar z$ at $(i,j)$. The left-hand side of \eqref{eq:nogood} therefore counts agreeing cells. Further, the left-hand side is $81$ exactly for $\bar z$, and at most $80$ for every different grid. All other Sudoku and clue constraints are retained. Since $\bar z$ is already a valid completion, absence of a different feasible grid is equivalent to uniqueness.
\end{proof}

It is worth noting that the earlier maximisation objective in \cite{williams2026} assigns coefficient $-1$ to reference assignments and $+1$ to all others. A grid differing in $m$ cells therefore attains an objective value of $-81+2m$, so an optimal value of $-81$ proves uniqueness. In contrast, the exclusion constraint \eqref{eq:nogood} allows both completion and uniqueness certification to be treated as feasibility problems with a constant objective, where uniqueness is established when the exclusion model is proved infeasible.

\subsection{Natural cell values and explicit linearisation}
The natural variables describe the board directly: $z_{ij}\in N$ is the value in cell $(i,j)$. To enforce the Sudoku rules, we must enforce that any two cells in a common row, column or block must have different values. 

Let $E$ be the set of all pairs of distinct cells sharing a row, column or block, with each pair included only once. We order the cells in each pair according to their row-major positions, writing (u<v), and introduce a binary variable $y_{uv}$. The two possible orderings of their values are represented by
\begin{align}
 z_u-z_v&\ge1-9y_{uv},\label{eq:pair1}\\
 z_v-z_u&\ge1-9(1-y_{uv}),\label{eq:pair2}
\end{align}
where the binary variable $y_{uv}$ chooses which value is larger. If $y_{uv}=0$, then \eqref{eq:pair1} requires $z_u\ge z_v+1$, while \eqref{eq:pair2} adds no restriction to the cell bounds. If instead $y_{uv}=1$, then the roles are reversed. Conversely, any two unequal integer values permit one of these choices. The inequalities therefore require exactly that the two values differ. 

There are $810$ distinct pairs. The initial $27\binom92$ count includes $81$ row--block duplicates and $81$ column--block duplicates. The completion model has $81$ cell variables, $810$ auxiliary binaries, $1620$ inequalities, and clue bound fixings. The raw baseline uses $M=9$. We later consider strengthened variants with redundant equalities and common singles preprocessing, all of which preserve the same integer completions.

To exclude the reference completion, at least one cell must take a smaller or larger value than its reference value. The binary variables $d^-_{ij}$ and $d^+_{ij}$ indicate these two possibilities, where the signs indicate a smaller or larger digit, respectively. We impose
\begin{align}
 z_{ij}&\le\bar z_{ij}-1+(10-\bar z_{ij})(1-d^-_{ij}),\label{eq:less}\\
 z_{ij}&\ge\bar z_{ij}+1-\bar z_{ij}(1-d^+_{ij}),\label{eq:greater}\\
 d^-_{ij}+d^+_{ij}&\le1 &&i,j\in N,\\
 \sum_{i,j\in N}(d^-_{ij}+d^+_{ij})&\ge1.\label{eq:some}
\end{align}
The first two inequalities are imposed for all cells. Note that, since the reference completion $\bar z$ is already known, its values are fixed constants rather than decision variables. The constraints are therefore linear, and the resulting formulation is mixed-integer linear. The coefficients are chosen so that an inactive inequality imposes no additional restriction beyond the cell bounds.

\begin{proposition}
For integral cell values $z\in N^{9\times9}$, binary variables satisfying \eqref{eq:less}--\eqref{eq:some} exist if and only if $z\ne\bar z$.
\end{proposition}
\begin{proof}
Notice that an active $d^-_{ij}$ forces $z_{ij}<\bar z_{ij}$, while an active $d^+_{ij}$ forces $z_{ij}>\bar z_{ij}$. Observe that at least one must be active by \eqref{eq:some}, so equality of the grids is impossible. Conversely, if a cell differs, activate its appropriate indicator and set all others to zero. An inactive upper inequality reduces to $z_{ij}\le9$ and an inactive lower inequality to $z_{ij}\ge1$. Thus all inequalities hold. The constants are tight because the inactive inequalities must permit the endpoints $9$ and $1$, respectively.
\end{proof}

\subsection{Exclusion through the existing order variables}\label{sec:order}
The preceding exclusion constraint introduces additional binary variables. It turns out that it is possible to exclude the reference completion using the order variables already present in the natural formulation. Recall that, in an integer solution, the inequalities \eqref{eq:pair1}--\eqref{eq:pair2} imply that $y_{uv}=1$ if and only if $z_u<z_v$.

Let $\bar y$ be the ordering induced by the reference grid $\bar z$, and let $E_R\subset E$ contain only those pairs of cells that share a row. There are $324$ such pairs. We can then impose
\begin{equation}
 \sum_{e\in E_R:\bar y_e=0}y_e+
 \sum_{e\in E_R:\bar y_e=1}(1-y_e)\ge1.\label{eq:orderexclude}
\end{equation}
Notice that this constraint requires at least one ordering to differ from that of the reference grid.

\begin{proposition}\label{prop:order}
The constraint~\eqref{eq:orderexclude} excludes exactly the reference Sudoku grid among integer feasible solutions. The same is true with $E$ in place of $E_R$.
\end{proposition}
\begin{proof}
Each row is a permutation of $N$, so the value in a cell is one plus the number of smaller values in that row. Hence the row orderings determine every cell value. Agreement with all of $\bar y$ on $E_R$ is equivalent to agreement with $\bar z$. Conversely, a grid has only one ordering vector because peer values differ. Excluding either the row vector or the full vector therefore excludes exactly $\bar z$.
\end{proof}

We therefore have three ways to exclude the reference completion: using additional binary variables to indicate differences in cell values, using the existing order variables for all pairs in (E), or using only the row pairs in $E_R$. Although the latter two approaches introduce no new binary variables, they need not require less computational effort, since their relaxations and the deductions available to the solver may differ. We compare all three approaches experimentally.

\subsection{What the continuous relaxations retain}\label{sec:relaxation}
The pairwise constraints enforce the Sudoku rules when the variables are integral. Their continuous relaxation behaves very differently. Once the ordering variables may take fractional values, every choice of cell values between one and nine that agrees with the clues can be extended to a feasible relaxed solution. This remains true for each of the three raw certification formulations. Note that by projecting onto the cell values $z$, we mean considering which values of $z$ are feasible when the auxiliary variables are allowed to take any values satisfying the relaxed constraints. The next proposition makes this loss of information precise.
\begin{proposition}\label{prop:box}
Before strengthening or solver reductions, the projections onto $z$ of the raw natural completion LP and each raw certification LP (cell-witness, row-order and full-order exclusion) are all exactly
\[
 B(P)=\{z\in[1,9]^{81}:z_{ij}=g_{ij}\text{ for }(i,j)\in\C(P)\}.
\]
\end{proposition}
\begin{proof}
Fix $z\in B(P)$ and write $\Delta=z_u-z_v\in[-8,8]$. The permitted ordering variable satisfies
\[
 y_{uv}\in[(1-\Delta)/9,(8-\Delta)/9]\cap[0,1].
\]
The first interval has ordered endpoints, its lower endpoint is at most one, and its upper endpoint is at least zero. The intersection is nonempty. Each pair has its own variable, so completion imposes no further restriction on $z$.

For the cell-witness exclusion, choose a cell $v$ with $\bar z_v=5$, which exists in every reference grid. Set
\[
 d^-_v=(9-z_v)/8,\qquad d^+_v=(z_v-1)/8,
\]
and all other witnesses to zero. These values lie in $[0,1]$ and sum to one. At $v$, the upper and lower witness inequalities have respective slacks $3(9-z_v)/8$ and $3(z_v-1)/8$, both nonnegative. Inactive inequalities elsewhere reduce to the cell bounds. Hence every $z\in B(P)$ also has a feasible cell-witness extension.

For either ordering exclusion, take $y_{uv}=(8-z_u+z_v)/16$. These values satisfy every relaxed pair inequality. In any row, list the cells as $v_1,\ldots,v_9$ in increasing reference-digit order. The disagreement terms for consecutive cells sum to
\[
 \sum_{k=1}^{8}\frac{8+z_{v_k}-z_{v_{k+1}}}{16}
 =\frac{64+z_{v_1}-z_{v_9}}{16}\ge\frac{56}{16}>1.
\]
These terms occur in both ordering no-goods and every remaining term is nonnegative. Thus both exclusions are satisfied. The cell bounds and clue fixings give the reverse inclusion in every case. Appendix~\ref{app:math-projection} spells out the pair-inequality substitution.
\end{proof}
The raw LP relaxation therefore retains no restriction on the cell values beyond their bounds and original clues. In particular, it remains feasible even when the integer certification model is infeasible. This statement concerns the submitted relaxation before solver reductions, cuts or other strengthening. The auxiliary constraints can still lead to different deductions during the solve, so Proposition~\ref{prop:box} does not predict an LP iteration count or runtime.

These are constant-objective feasibility models. A conventional objective-value integrality gap consequently tells us little about their strength. The relevant question is which points the relaxation admits, and what restrictions remain after projecting onto the cell values.

One way to strengthen the LP relaxation is to add the $27$ unit sums $\sum_{u\in V}z_u=45$, one for each row, column and block $V$. As discussed in our earlier teaching paper \citep{williams2026}, these equalities cannot replace the all-different. For example, assigning five to every cell satisfies the sums but not the Sudoku rules. However, when added to the existing formulation, these equalities preserve every valid Sudoku completion while eliminating some fractional solutions admitted by the raw LP relaxation. Further discussion of these equalities and their limitations can be found in P\"oppel \citep{poeppel2025}.

A stronger tested variant adds, as well, the $81$ row-rank equalities
\begin{equation}
 z_{ij}=1+\sum_{t<j}y_{(i,t),(i,j)}
             +\sum_{t>j}(1-y_{(i,j),(i,t)}).\label{eq:ranks}
\end{equation}
They are valid because their right-hand side counts smaller row entries. They couple the formerly independent ordering variables to the cell values. Summing the rank equations in one row yields its sum constraint, because each ordering variable contributes $y+(1-y)=1$. Thus the row sums are redundant even in this strengthened relaxation; they are retained to keep the unit-sum block fixed across variants. The experiments distinguish the raw model, unit sums, and unit sums plus ranks.

The assignment exclusion can also be strengthened by using a simple property of completed grids: an alternative cannot differ in just one, two or three cells.
\begin{proposition}\label{prop:four}
Distinct completed Sudoku grids differ in at least four cells. Thus, replacing $80$ by $77$ in \eqref{eq:nogood} preserves the integer feasible set.
\end{proposition}
\begin{proof}
Form a bipartite graph whose edges are changed cells and whose vertices are their rows and columns. An affected row or column cannot have exactly one changed entry, since both versions contain each digit once. Every nonisolated vertex therefore has degree at least two. A nonempty finite graph with this property contains a cycle, and a bipartite simple cycle has at least four edges. Thus a different grid agrees in at most $77$ cells.
\end{proof}

\subsection{High-level all-different model}
The CP-SAT model uses integer cell values and imposes an all-different constraint on each row, column and block, ensuring that every digit appears exactly once. The given clues are fixed directly. To test uniqueness, we exclude the reference completion \(\bar z\) using CP-SAT's forbidden-assignment constraint, which prevents the solver from returning that completed grid. CP-SAT determines how these constraints are represented and processed internally \citep{ortools}, so we do not assume that it uses any particular all-different propagation algorithm.

Note that both the assignment and natural cell-value MILP formulations use the same HiGHS solver \citep{highs}, allowing us to compare the effect of the mathematical formulation, while comparisons with CP-SAT involve both a different formulation and a different solver.

\subsection{When completion already proves uniqueness}\label{sec:singles}
A fresh exclusion solve is a well-defined certification protocol, but a sequence of forced assignments may already prove uniqueness. Every completion must respect each naked or hidden single, so a complete sequence leaves no choice for any cell. These deductions also apply to fractional assignment variables. The following proposition gives the singles special case of Provan's theorem \citep[Theorem~2]{provan2009}.
\begin{proposition}[Provan, singles special case]\label{lem:singles}
If naked and hidden singles complete a solvable puzzle, then its completion is unique. Moreover, the standard assignment model's continuous relaxation consists of that single completed assignment vector.
\end{proposition}
\begin{proof}
A fixed assignment and nonnegativity force competing assignments in its cell and peer digit constraints to zero. A naked single then follows from the cell equality; a hidden single follows from the relevant unit--digit equality. Each such deduction fixes a variable to one in every feasible point of the continuous relaxation. Induction along the complete trace fixes all assignments. An existing completion supplies a feasible point, so the relaxation and its integer subset are both singletons.
\end{proof}
Thus a complete singles sequence both identifies the unique completion and fixes every point of the assignment LP relaxation to that completion. Adding the exclusion makes this relaxation infeasible. For such puzzles, feasibility of the assignment LP determines completion, while infeasibility of its reference-exclusion LP establishes uniqueness. This mathematical property does not prescribe the computational operations performed by a particular solver.

The contrast with Proposition~\ref{prop:box} is direct. The assignment relaxation retains these forced choices, whereas the raw natural relaxation still admits the entire clue-consistent box. Applying the same singles closure before either MILP model supplies the forced values explicitly; we measure that preprocessing separately.

The same deductions extend to generalised $n^2\times n^2$ Sudoku. A complete singles sequence fills at most $n^4$ cells, and candidate sets and unit positions can be checked in polynomial time in the grid size. Puzzles completed by these rules can therefore be completed and certified in polynomial time. This is a consequence for that restricted class, not a claim about generalised Sudoku as a whole.

\subsection{Minimum clue-deletion margin}\label{sec:margin}
We now consider a related question: given a Sudoku puzzle with a unique completion, how many clues must be removed before another completion becomes possible? We call this number the \textit{minimum clue-deletion margin}. For any competing completed grid, every original clue with which it disagrees must be removed. Minimising that number over competing grids gives
\begin{equation}
 \rho(P)=\min_{z\in\Z,\ z\ne\bar z}\ \sum_{(i,j)\in\C(P)}\ind\{z_{ij}\ne\bar z_{ij}\}.\label{eq:rho}
\end{equation}
The candidate grids in this definition need not satisfy the original clues. In the assignment model, minimise $\sum_{(i,j)\in\C(P)}(1-x_{ij,\bar z_{ij}})$ subject to the completed-grid rules and \eqref{eq:nogood}, \emph{without} imposing \eqref{eq:clues}.

The following proposition formalises this interpretation.
\begin{proposition}\label{prop:deletion-margin}
For a uniquely completable puzzle, $\rho(P)$ is the minimum number of original clues whose deletion permits an alternative completion.
\end{proposition}
\begin{proof}
For a competing grid $z$, write $S(z)=\{(i,j)\in\C(P):z_{ij}\ne\bar z_{ij}\}$. Deleting $S(z)$ permits the competing grid $z$. Conversely, any deletion set $R$ that permits $z$ must contain $S(z)$, and hence $|R|\ge|S(z)|\ge\rho(P)$. Thus \eqref{eq:rho} gives exactly the minimum deletion size.
\end{proof}
It is also a minimum-cardinality correction-set value, with Sudoku and exclusion constraints fixed and clue equalities removable.

The optimisation formulation gives the minimum number of clues whose removal destroys uniqueness; unavoidable sets explain how clue placement determines this number.

A structural interpretation follows from the established theory of unavoidable sets \citep{mcguire2014,tjusila2024}. An \emph{unavoidable set} of a completed grid is a set of cells that cannot all be left without clues if that grid is to be determined uniquely: fixing every cell outside the set still permits an alternative completion. A minimal unavoidable set contains no smaller such set. Let $\mathcal U_{\min}(\bar z)$ denote their family.
\begin{proposition}[Unavoidable-set characterisation]\label{prop:margin}
For a puzzle $P$ with unique completion $\bar z$,
\[
 \rho(P)=\min_{U\in\mathcal U_{\min}(\bar z)}|\C(P)\cap U|.
\]
\end{proposition}
The proof is given in Appendix~\ref{app:math-margin}. Each competing grid differs from $\bar z$ on an unavoidable set, and every minimal unavoidable set is realised by such a difference. The margin is therefore determined by the \emph{least-covered minimal unavoidable set}. Total clue count alone does not describe this coverage.

Uniqueness implies $\rho(P)\ge1$. It should be noted that a margin of one means that there exists at least one clue whose removal destroys uniqueness, whereas minimality requires that removing any single clue destroys uniqueness. Thus, every minimal puzzle has margin one, but the converse is not necessarily true.
\begin{corollary}\label{cor:rho}
If $P'$ adds one correct clue to $P$, then
\[
 \rho(P)\le\rho(P')\le\rho(P)+1.
\]
The margin is invariant under every Sudoku-preserving symmetry applied simultaneously to the clues and reference grid.
\end{corollary}
\begin{proof}
Every competing grid's objective increases by either zero or one when the clue is added; taking minima gives both inequalities. Such a symmetry bijects completed grids and their difference supports, preserving the cardinality of the intersection with the transformed clue set.
\end{proof}
These properties concern the structural robustness of uniqueness. They imply no monotonic relationship between the margin and computational effort. The margin counts disagreements at clue positions, rather than numerical differences between arbitrary digit labels; Appendix~\ref{app:math-margin} makes this distinction precise.


\section{Structural features}\label{sec:features}
To investigate why puzzles with the same number of clues can require different computational effort, we examine how the clues are distributed and what can be deduced directly from them. These structural features are computed independently of the experimental solvers. The main models use the original clues, while the separately tested propagation variants also incorporate deductions from naked and hidden singles.

Let $r_i$, $c_j$ and $b_t$ denote the number of clues in row $i$, column $j$ and block $t$, respectively. Let $d_k$ count how many times digit $k$ appears among the clues. For each of these four count vectors, we calculate the minimum, maximum, range, population variance and entropy. Since each digit can appear at most nine times in a completed Sudoku grid, every entry in these vectors is between zero and nine. We calculate entropy as
$$
-\sum_{v_a>0}(v_a/C)\log(v_a/C)
$$
where $v_a$ denotes an entry in the corresponding count vector and $C=|\C(P)|$ is the total number of clues.

For a row vector, write $I_R=\sum_i(r_i-C/9)^2$, with analogous column and block quantities. Note that $I_R$ measures how unevenly the clues are distributed across the rows, with larger values indicating greater imbalance. However, the possible values of this imbalance measure depend on the clue count, because no row can contain more than nine clues. We therefore scale it using its smallest and largest possible values at each clue count.

For $C=9q+s$, $0\le s<9$, the exact extrema over integer count vectors are
\[
 I_{\max}(C)=81q+s^2-C^2/9,\qquad I_{\min}(C)=s(9-s)/9.
\]
The maximum fills units to capacity before starting another, while the minimum distributes clues as evenly as possible. Appendix~\ref{app:math-imbalance} gives an exchange proof. 

Define the normalised imbalance index
\[
 J_R=\frac{I_R-I_{\min}(C)}{I_{\max}(C)-I_{\min}(C)}
\]
when $I_{\max}(C) \ne I_{\min}$. If instead $I_{\max}(C)=I_{\min}$, then we set $J_R = 0$ and record this degenerate case separately. Define $J_C,J_B,J_D$ similarly and 
$$
J_{\mathrm{grid}}=\frac{J_R+J_C+J_B}3.
$$ 
It should be noted that $J_R$, $J_C$, and $J_B$ measure clue-count imbalance across rows, columns, and blocks, respectively, while $J_D$ measures imbalance in the frequencies of the nine given digits. Each index is in $[0,1]$ over bounded integer count vectors; simultaneous attainment or attainment by unique puzzles is not assumed. At eighty clues the distribution is forced, and the flag avoids calling its unavoidable imbalance unusually uneven. Counts of ``nearly complete'' units are also recorded, but are not all entered into regression.

For each initially empty cell, $S_{ij}$ contains digits absent from its given peers. These are \emph{direct-elimination} candidates, before iterative propagation. We record their mean and maximum size, size frequencies, total mass $M=\sum|S_{ij}|$, and log-ambiguity $H_0=\sum\log|S_{ij}|$. The last quantity is a descriptive sum of log-domain sizes, not the Shannon entropy of the dependent solution space. The compact regression uses $H_0/(81-C)$.

Propagation repeatedly applies the naked and hidden singles introduced in Section~1. Each round recomputes candidates from the current grid, collects all forced cell values from the same state, and applies them simultaneously. The process stops at a fixed point. We record initially forced cells, rounds, newly filled cells, the proportion of original blanks filled, unresolved cells and remaining candidate mass. Only these two rules are used; Appendix~\ref{app:features} specifies the procedure step by step.

The residual graph has unresolved cells as vertices and edges between cells sharing a row, column or block. The feature table includes vertices, edges, mean and maximum degree, components and largest component. When no cells remain unresolved, all these graph summaries are set to zero. The final compact model uses unresolved-cell count; graph order is the same quantity, and several other graph summaries largely track the number of unresolved cells. These summaries are recorded for description; the regression uses the smaller interpretable set.


\section{Computational methodology}\label{sec:method}
\subsection{Execution and validation}
Three puzzle collections serve different purposes. The \NInstances-puzzle benchmark balances every exact clue count and supports the main descriptive and predictive analyses. The \ExternalN-puzzle Norvig collection provides an additional validation check after removing puzzles equivalent under Sudoku symmetries. The fixed \HardN-puzzle Tdoku sample is drawn from a historically highly rated collection to investigate the relationship between completion and certification effort on puzzles where the more efficient models record nonzero search counters. The latter two collections are neither balanced by clue count nor representative of all Sudoku puzzles.

Table~\ref{tab:design} summarises the full benchmark alongside the smaller diagnostic and validation experiments. These additional experiments investigate specific aspects of solver behaviour, rather than providing results representative of the full benchmark. For each experiment comparing completion and uniqueness certification, both tasks are performed on the same puzzles. The minimum clue-deletion margin is evaluated separately as an optimisation problem.

\begin{table}[H]\centering\fontsize{8.5}{10.5}\selectfont\setlength{\tabcolsep}{3pt}
\caption{Experimental coverage. Note that ``instances'' includes relabelled and nested puzzles, which are not necessarily independent. ``Calls'' counts retained solver runs, excluding abandoned recovery attempts. Diagnostic experiments use a ten-second limit, with a five-second limit for the nested raw-natural runs and a sixty-second limit for the matched cap extension and opening presolve checks. Pilot runs are excluded, and margin optimisation requires one call per instance. \vspace{2.0mm}}\label{tab:design}
\begin{adjustbox}{max width=\linewidth}
\begin{tabular}{>{\raggedright\arraybackslash}p{110pt}>{\raggedright\arraybackslash}p{105pt}rrr}
\toprule
Experiment & Selection & Instances & Calls & Finished \\
\midrule
Assignment and CP-SAT & Exact-count benchmark & 32,000 & 128,000 & 128,000 \\
Raw natural & Exact-count benchmark & 32,000 & 64,000 & 64,000 \\
Formulation diagnostics & Random selection + cases & 90 & 1,620 & 1,597 \\
Digit relabelling & Selected base puzzles & 252 & 1,512 & 1,473 \\
External validation & Norvig lists, deduplicated & 105 & 630 & 619 \\
Harder external validation & Fixed Tdoku sample & 64 & 768 & 767 \\
Clue-deletion margin & Stratified benchmark subset & 320 & 320 & 320 \\
Nested: assignment and CP & Shared unique bases & 2,560 & 10,240 & 10,240 \\
Nested: raw natural & Same derived sequences & 2,560 & 5,120 & 5,034 \\
Presolve on/off & Benchmark + harder sample & \PreAllN & \PrePrimaryCalls & \PrePrimaryFinished \\
Matched cap extension & Any unfinished puzzle & \PreExtendedN & \PreExtendedCalls & \PreExtendedFinished \\
Opening presolve cases & Figure~\ref{fig:opening} & \PreFigureCases & \PreFigureCalls & \PreFigureFinished \\
\bottomrule
\end{tabular}
\end{adjustbox}
\end{table}

The primary campaigns use direct Python interfaces to HiGHS 1.15.1 or OR-Tools 9.15.6755 on an Apple M4 machine with 16 GB memory and macOS 27. Each solver uses one thread and seed zero in the primary runs. Standard presolve and solver heuristics remain enabled in the primary campaigns. The matched sensitivity experiment varies the HiGHS presolve option alone, as described in Appendix \ref{app:presolve}. HiGHS uses feasibility tolerances of \(10^{-7}\) and zero requested absolute and relative MIP gaps. Completion receives only the givens, with no archived-solution hint. Certification is a fresh solve receiving the givens and the known reference, without a retained basis, search state or completion warm start. The CP run stops after its first solution when one exists.

A pilot sampled one puzzle at every count before the full campaign. Assignment and CP-SAT then received a ten-second limit per task. We first ran the raw natural model with a one-second limit, alongside a separate ten-per-count subset with a five-second limit. Unfinished full-screen calls were restarted from scratch with limits of ten, sixty and, where needed, three hundred seconds. The primary analysis uses the first run that finishes. Work from earlier unsuccessful attempts is recorded separately and is not included in that run's effort count. If no stage finishes, the outcome remains unresolved; neither its time limit nor its partial counter is treated as completed effort. Appendix~\ref{app:recovery} records the stage coverage and recovery sensitivity.

Independent puzzle solves run concurrently, while each solver remains single-threaded. The worker caps and execution records are given in Appendix~\ref{app:solver}. We use counters as the primary outcomes because wall times and timeout status can also depend on machine load. Repeated serial batches separately compare feature extraction, singles preprocessing and CP solve cost.

The formulation diagnostic sample consists of five randomly selected puzzles at each count from seventeen through thirty, and at forty, sixty and eighty clues. We add the illustrative cases and one nested pair selected because the original run recorded a large increase. This is a diagnostic, deliberately enriched sample; its aggregate proportions are not estimates for the full benchmark. Each instance receives the raw, unit-sum and rank-strengthened natural models, both ordering exclusions, both assignment exclusions, and the same singles closure before each original MILP. Every variant has the same ten-second cap. Preprocessing time is recorded separately. A standalone presolve inspection uses a separate copy and the settings in Appendix~\ref{app:solver}; it is not charged to, or reused by, the measured solve.

For digit sensitivity we select two puzzles in each count--residual class at counts seventeen, twenty and twenty-four, separating puzzles that singles complete from those with cells still unresolved. Each base has its identity labelling and twenty fixed random digit permutations. Givens and reference are relabelled together; the solver seed remains zero. We check feature invariance and compare raw natural, row-order exclusion and assignment models. All data-dependent example selection is disclosed and kept separate from the random selections.

We additionally use Norvig's public \texttt{top95} and \texttt{hardest} lists \citep{norvig2006}. An independent integer enumerator checks uniqueness. A coloured incidence-graph isomorphism test checks exact equivalence under standard Sudoku symmetries; its construction and justification are in Appendix~\ref{app:symmetry}. We check against the full benchmark and remove one symmetry duplicate within the external lists. This separately sourced, curated collection receives the same two tasks and ten-second caps under assignment MILP, raw natural MILP and CP-SAT.

To obtain more observations with nonzero search counters, we freeze a uniform sample of \HardN\ source lines from the \HardSourceN-record forum-hardest 11+ file distributed with Tdoku \citep{tdoku}. Its historical Sudoku Explainer rating is a source-selection criterion, not our computational outcome. Selection uses seed 20261007 before any runs of our models. An independent integer enumerator verifies every selected puzzle, and exact symmetry screening finds no overlap with the benchmark, Norvig collection or other selected records. The full fixed sample receives assignment, CP-SAT, raw natural, common-singles natural, rank-strengthened and row-order models, each with the same ten-second cap and four single-threaded workers. No puzzle is removed or replaced because of a solver outcome. Appendix~\ref{app:validation} specifies the source selection.

Every returned completion is checked against the givens, every unit and the archived solution. A second completion in a benchmark certification run is an error, and a timeout never establishes uniqueness. Independent integer enumeration and CP margin checks complement the formulation tests; Appendix~\ref{app:validation} gives their coverage and explains the checks on returned solutions and recorded results.

\subsection{Effort outcomes and statistical analysis}
HiGHS records processed branch-and-bound nodes, LP iterations, solver and external elapsed time, model dimensions, available presolve reductions, status, bounds and residual statistics. The LP counter sums iterations reported across the solve, including those attributed to separation and heuristics. It counts iterations, not how many LP problems were solved, and does not measure all work done in presolve or propagation. A zero LP count therefore does not imply zero solver effort. Root objectives in a zero-objective feasibility model do not provide a useful hardness scale. A processed-node count of one can describe the root alone; a positive node counter is not by itself evidence of branching. No deterministic HiGHS work-unit counter is available through this interface. CP-SAT records branches, conflicts, LP iterations and its reported deterministic time. Deterministic time is an implementation-specific work proxy, not elapsed seconds or a common unit shared with HiGHS.

The primary distribution plots retain zeros and unresolved outcomes. Numerical modelling uses $\log(1+D)$ so that zero counts remain defined. Rank correlations are undefined when either input is constant.

A high overall correlation can arise partly because both tasks record zero LP iterations on many puzzles, or because both counters decrease as clue count increases. We therefore examine the relationship after excluding pairs with shared zeros and after comparing puzzles with the same clue count. We report both pairs with at least one positive counter and pairs with both counters positive, alongside the full population and generation-regime comparisons.

For the within-count comparison, we rank puzzles separately at each exact clue count, subtract the mean rank in each group, and correlate the centred midranks. Larger groups with more rank variation contribute more to this pooled statistic; it is not an equally weighted average of the separate Spearman correlations. Intervals resample puzzle pairs within clue count and recompute the ranks. To compare two models' task associations, each bootstrap draw uses the same puzzle IDs for both models and subtracts their coefficients. This paired contrast is computed separately for ordinary and within-count correlations. Appendix~\ref{app:prediction} specifies the resampling, and Appendix~\ref{app:comparisons} gives the supporting comparisons.

A correlation describes whether the two tasks tend to rank puzzles similarly; it does not tell us whether certification usually requires more iterations. We also report the paired difference $\Delta=\log(1+D_{\mathrm{cert}})-\log(1+D_{\mathrm{find}})$, its quantiles, and the numbers of cheaper, equal and more expensive certification runs. For pairs positive in both tasks, we additionally report the median of the individual ratios $D_{\mathrm{cert}}/D_{\mathrm{find}}$. A ratio of marginal medians is not this paired statistic.

We use a logistic model to predict whether a solve records any LP iterations, including unfinished runs whose recorded work is already positive. Regressions then predict $\log(1+D)$, evaluated separately across all finished runs and among runs with positive counts. The positive-only scores assume that positive work has already been observed; they do not assess a combined predictor for a new puzzle whose zero/positive status is unknown. Each model has a separate baseline $\alpha_C$ for every clue count, formally a count-specific intercept. We then add clue distribution, candidate ambiguity and the number $U$ of cells left unresolved after singles:
\begin{align*}
 M_0 &: \alpha_C,\\
 M_{\mathrm{grid}} &: \alpha_C+\beta_G J_{\mathrm{grid}},\qquad
 M_{\mathrm{digit}} : \alpha_C+\beta_D J_D,\\
 M_1 &: \alpha_C+\beta_G J_{\mathrm{grid}}+\beta_D J_D,\\
 M_2 &: M_1+\beta_H H_0/(81-C),\\
 M_3 &: M_2+\beta_U U.
\end{align*}
The separate grid and digit models distinguish clue placement from digit balance. Numerical coefficients are regularised within each training fold, while regression baselines for clue count are unpenalised. Appendix~\ref{app:prediction} specifies the regression and logistic penalties and reports all retained populations.

Five fixed folds are stratified by exact clue count, and all feature blocks use the same splits. A regression population retains only counts with at least five eligible observations. We report held-out $R^2$, mean absolute error (MAE), and paired changes from $M_0$. Their intervals resample the held-out predictions within clue count. They measure evaluation-sample uncertainty conditional on the fitted folds, rather than uncertainty from refitting or transfer to a new population. The symmetry audit in Appendix~\ref{app:symmetry} finds no equivalent benchmark pair; related generators can still create other dependence.

Linear predictions of a nonnegative log-effort outcome can be negative. Alongside the original fits, we evaluate $\max(0,\widehat y)$ on exactly the same held-out predictions, without refitting or changing folds. Projection onto the nonnegative half-line cannot increase absolute or squared error when the observed outcome is nonnegative. This checks whether extrapolation below zero explains the reported MAE trade-off.

For CP deterministic work $d$, the outcome is $\log(1+10^6d)$. We compare adding $\ind\{U>0\}$ with adding $U$, repeat the fit within $U>0$, and assess their held-out scores. The scaling is part of the defined outcome, not a conversion to wall-clock seconds. For the minimum clue-deletion margin we retain the fixed original sample and its exact optima, and refit $M_4=M_3+\beta_\rho\rho(P)$ on the common eligible records using the same unpenalised count intercepts. Natural-model fits use the recovered certification outcomes. The analysis first checks whether margin variation occurs within counts where the effort outcome varies; without that overlap, a fitted coefficient cannot show whether the margin distinguishes positive LP effort beyond clue count.

\subsection{Matched clue additions}
We select twenty unique $17$-clue base puzzles and generate two seeded random orders of their blank cells. Each sequence records every puzzle from $17$ to $80$ clues, adding the archived value at one cell per step. Adding a correct clue preserves uniqueness: the new completion set is a subset of the old singleton and still contains $\bar z$. Every base puzzle is checked afresh by independent enumeration. These derived instances are stored separately from the benchmark.

We measure the three formulations on these sequences, using the primary assignment/CP limits and a five-second natural-model limit. Differences $\Delta D_C=D(P_{C+1})-D(P_C)$ use adjacent completed pairs only. Unresolved pairs are counted separately. The two orders from one base share many clues and the same solution; bootstrap intervals for increase rates resample base puzzles, not individual steps. The integer completion set remains exactly the same singleton throughout. Its representation and relaxation change, and no monotonicity of a solver's operations follows. The experiment tests that behaviour empirically.


\section{Results}\label{sec:results}
\subsection{Clue count, variation and zero search counters}
All main-benchmark completion and certification runs finish. Staged recovery resolves the \InitialUnfinished\ task calls left unfinished by the original screen, giving \RecoveredFindN\ completion results and \RecoveredCertN\ infeasibility results for the raw natural model. The results therefore cover every benchmark instance, including the low-count tail omitted by a completed-only analysis of the initial screen. Appendix~\ref{app:recovery} reports coverage at each recovery stage.

Figure~\ref{fig:distribution} shows the conditional effort distribution at every clue count. With presolve enabled, the raw natural model records positive LP work on \PositiveFindN\ completion calls and \PositiveCertN\ certification calls. All runs from thirty-one through eighty clues report zero LP iterations. Positive LP counts are therefore concentrated in a small part of the clue-count range. An average across all puzzles would hide both the large counts in that range and the many zero counts elsewhere.

\begin{figure}[ht!]\centering\includegraphics[width=0.85\linewidth]{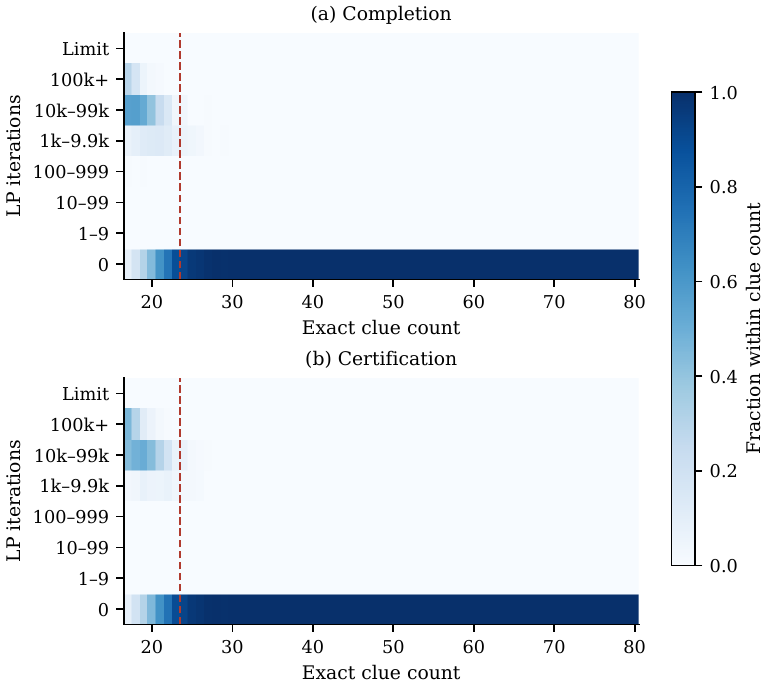}
\caption{Full conditional distributions of raw natural-model LP effort. Panels (a) and (b) show completion and certification. Each column divides all calls at one exact clue count into the displayed effort bins. The limit row is empty after recovery. The dashed line marks the change in the source generator; it is not an estimated difficulty threshold.}\label{fig:distribution}
\end{figure}
\begin{figure}[ht!]\centering\includegraphics[width=0.85\linewidth]{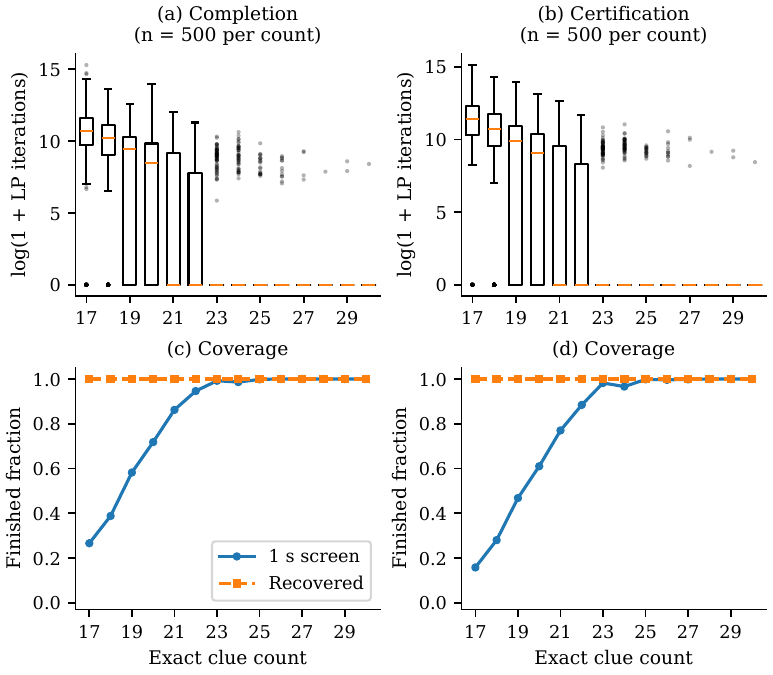}
\caption{The active count range after recovery. Panels (a) and (b) show boxplots of $\log(1+\mathrm{LP})$, with finished sample sizes stated in each panel. Whiskers extend to the most extreme observations within $1.5$ interquartile ranges; points beyond are shown individually. Panels (c) and (d) compare coverage in the one-second screen and recovered population. The recovered boxes include every instance in each stratum.}\label{fig:lowcounts}
\end{figure}

Consider the seventeen-clue group in Figure~\ref{fig:lowcounts}. Median LP effort is \SeventeenFindMedian\ for completion and \SeventeenCertMedian\ for certification, while the maxima are \SeventeenFindMaximum\ and \SeventeenCertMaximum. The same group also contains \SeventeenFindZero\ zero-LP completion calls and \SeventeenCertZero\ zero-LP certification calls. Exact clue count thus leaves substantial variation under this formulation. The longer runs matter here because summaries based only on finished runs from the initial one-second screen omitted the slower cases.

With presolve enabled, the assignment model records zero branch-and-bound nodes and zero LP iterations for both completion and certification across the entire benchmark. Similarly, CP-SAT records zero branches for both tasks. These counters therefore cannot distinguish the computational difficulty of the puzzles. However, zero counts do not mean that no computation takes place: the solver logs show reductions, propagation and, in some CP-SAT runs, probing. For example, the assignment completion model for \texttt{SUD-18-0268} is reduced to an empty feasible model during presolve, while its certification model is found infeasible during presolve. The corresponding CP-SAT logs also show probing despite recording zero branches.

Elementary singles complete \SinglesCompleteN\ puzzles. Note that Proposition~\ref{lem:singles} establishes exactness of the assignment LP relaxation for that subset, but does not explain every solver reduction on the remaining puzzles. Conversely, even among singles-complete puzzles, the raw natural model records positive LP work on \SinglesPositiveFind\ completion calls and \SinglesPositiveCert\ certification calls. A forced solution does not guarantee that every representation exposes the deductions efficiently.

\subsection{Completion and certification: concordance and exceptions}
For the raw natural model, puzzles requiring more LP iterations to complete generally also require more iterations to certify. This relationship remains strong after excluding pairs with shared zeros and comparing puzzles with the same number of clues.

Following the recovery of unfinished runs, completion and certification effort have a Spearman correlation of \RecoveredAllR\ across all \RecoveredAllN\ puzzles. However, this includes \BothZeroN\ pairs where both tasks record zero LP iterations. Excluding these pairs gives a correlation of \RecoveredActiveR\ for the \RecoveredActiveN\ puzzles with at least one positive counter, and \RecoveredBothR\ for the \RecoveredBothN\ puzzles with positive counters in both tasks. The corresponding count-stratified bootstrap intervals are $[\RecoveredActiveLow,\RecoveredActiveHigh]$ and $[\RecoveredBothLow,\RecoveredBothHigh]$. Figure~\ref{fig:tasks} illustrates these two groups.

For puzzles with positive LP iterations in both tasks, the correlation of centred within-count ranks is \ConditionalPositiveLPR\, with interval $[\ConditionalPositiveLPLow,\ConditionalPositiveLPHigh]$. This shows that the association persists even when comparing puzzles with the same number of clues. Using branch-and-bound nodes instead gives a correlation of \ConditionalPositiveNodesR\ across \ConditionalPositiveNodesN\ pairs with positive node counts in both tasks. Table~\ref{tab:conditionalregimes} provides comparisons across the two generation procedures, although the higher clue-count range contains relatively few positive pairs, leading to greater uncertainty.

\begin{figure}[ht!]\centering\includegraphics[width=0.8\linewidth]{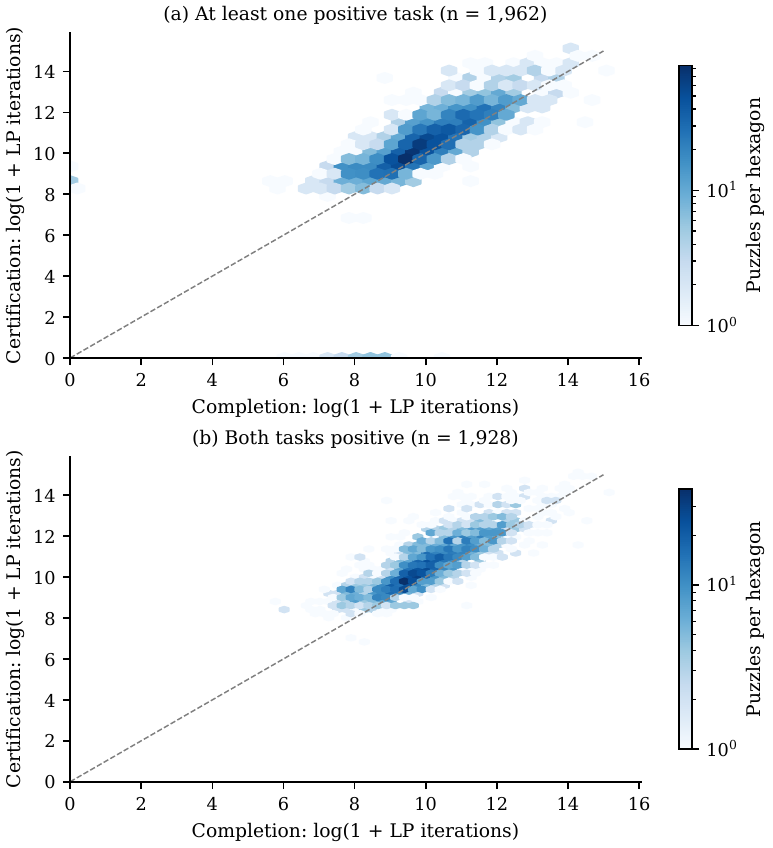}
\caption{Recovered raw natural-model completion and certification effort after excluding pairs with both LP counts zero. The colour bars give puzzles per hexagon on a logarithmic scale, with darker cells containing more puzzles; the dashed diagonal marks equal LP counts. Panel (a) includes task-discordant zero/positive pairs. Panel (b) retains only pairs positive in both tasks.}\label{fig:tasks}
\end{figure}

It turns out that certification is more expensive on \CertDearerN\ of the \RecoveredActiveN\ pairs with at least one positive LP counter, and cheaper on \CertCheaperN. The median paired log difference is \TaskDeltaMedian, with interquartile range $[\TaskDeltaQOne,\TaskDeltaQThree]$. Among the \RecoveredBothN\ pairs positive in both tasks, the median individual certification/completion ratio is \PairedPositiveRatio. Thus the positive association coexists with a typical certification premium and with exceptions in both directions. Table~\ref{tab:app-review_paired_cost} gives the paired cost summary.\par

Strong task association still allows individual puzzles to behave differently: \FindOnlyN\ puzzles have positive completion LP counts and zero certification LP counts, while \CertOnlyN\ show the reverse. The exclusion representation matters for interpreting such cases. Table~\ref{tab:cases} shows two opposite examples. For \texttt{SUD-20-0101}, reusing row-order variables reduces the certification LP count to zero. For \texttt{SUD-18-0268}, the same change introduces positive certification LP counts where the cell-witness version had none. Thus a task contrast can change without changing the puzzle, solver or excluded integer grid.\par

\begin{table}[!htbp]\centering\small
\caption{Selected cases: LP iterations in fresh diagnostic runs. $U$ is the number of cells unresolved by singles. ``Row'' changes only the certification exclusion; ``Rank'' adds unit sums and row-rank equalities to the cell-witness model. All displayed calls finished.}\label{tab:cases}
\begin{adjustbox}{max width=\linewidth}
\begin{tabular}{lrrrrr}
\toprule
Instance & $U$ & Raw find & Raw certify & Row certify & Rank certify \\
\midrule
SUD-17-0213 & 0 & 8,404 & 16,367 & 26,357 & 0 \\
SUD-20-0101 & 0 & 0 & 6,462 & 0 & 0 \\
SUD-18-0268 & 50 & 12,564 & 0 & 14,783 & 0 \\
\bottomrule
\end{tabular}
\end{adjustbox}
\end{table}

\subsection{Task concordance on harder external puzzles}
The separate validation collection contains \ExternalN\ unique puzzles after removing duplicates under Sudoku symmetries, with no equivalent puzzles found in the main benchmark. Both the assignment model and CP-SAT complete all runs, recording nonzero search counters on the same two puzzles. For the assignment model, completion processes only one branch-and-bound node in each case, which may represent the root node, whereas certification processes more than one. CP-SAT also records positive branch counts on these puzzles, although its branches measure a different type of solver operation. Under the same ten-second time limit, the raw natural model completes \ExternalNaturalFindDone\ completion runs and \ExternalNaturalCertDone\ certification runs.

The larger number of zero counters in the Norvig collection makes it difficult to compare how completion and certification effort are related. The harder sample provides a clearer comparison, with little correlation between the two tasks under assignment LP iterations, moderate correlation under raw natural LP iterations, and strong correlation under CP-SAT branches. These correlations describe the relationship between completion and certification within each model, rather than comparing computational effort directly across solvers.

The \HardN\ puzzles in the harder sample have between \HardMinClues\ and \HardMaxClues\ clues, and none can be completed using only naked and hidden singles. Table~\ref{tab:hard} reports the main correlations and their uncertainty, while Appendix~\ref{app:external} provides the clue-count frequencies, model coverage and additional comparisons.

\begin{table}[H]\centering\fontsize{8.5}{10.5}\selectfont\setlength{\tabcolsep}{3.5pt}
\caption{Completion--certification association on all \HardN\ harder-sample puzzles; every listed model finishes both tasks. Ordinary Spearman correlations and centred within-count rank correlations use the same task-specific counter within a row. The 95\% percentile intervals resample puzzle pairs within count. Counters across solver paradigms are not commensurable.\vspace{2.0mm}}\label{tab:hard}
\begin{adjustbox}{max width=\linewidth}
\begin{tabular}{llrrrr}
\toprule
Model & Counter & $r_S$ & 95\% interval & $r_{\mathrm{within}}$ & 95\% interval \\
\midrule
Assignment & LP & 0.031 & [-0.251, 0.273] & -0.106 & [-0.427, 0.225] \\
Raw natural & LP & 0.404 & [0.192, 0.572] & 0.265 & [-0.067, 0.527] \\
CP-SAT & Branches & 0.882 & [0.797, 0.930] & 0.831 & [0.647, 0.901] \\
\bottomrule
\end{tabular}
\end{adjustbox}
\end{table}

In the harder sample, assignment records positive LP work on \HardAssignmentFindPositive\ completion runs and all certification runs, and CP records positive branches throughout. Raw natural and rank-strengthened models also finish both tasks on every puzzle. Only row-order certification leaves an unresolved run. No elementary single is available initially, so common singles preprocessing changes no puzzle and its completed LP counters coincide with the raw model. Table~\ref{tab:hardcounts} gives the exact-count frequencies and positive assignment-completion counts.

Figure~\ref{fig:hardtasks} shows what the different correlations mean. A puzzle with a large completion count need not have a correspondingly large certification count in panel (a); the two counters track one another much more closely in panel (c). The panels use different counters and scales, so their vertical positions do not compare effort between solvers. Restricting assignment to the \SubmissionAssignmentPositiveN\ pairs positive in both tasks gives a correlation of \SubmissionAssignmentPositiveR, with interval $[\SubmissionAssignmentPositiveLow,\SubmissionAssignmentPositiveHigh]$. The completion zeros alone therefore do not account for its weak observed association.

After centring ranks within exact clue counts, CP retains a strong positive association. The assignment estimate remains weak and uncertain; the raw-natural estimate is positive but its interval also includes zero. Table~\ref{tab:hard} gives these adjusted estimates and intervals alongside the ordinary coefficients. The unequal count strata supply unequal information to the pooled within-count statistic.

\begin{figure}[ht!]\centering\includegraphics[width=0.8\linewidth]{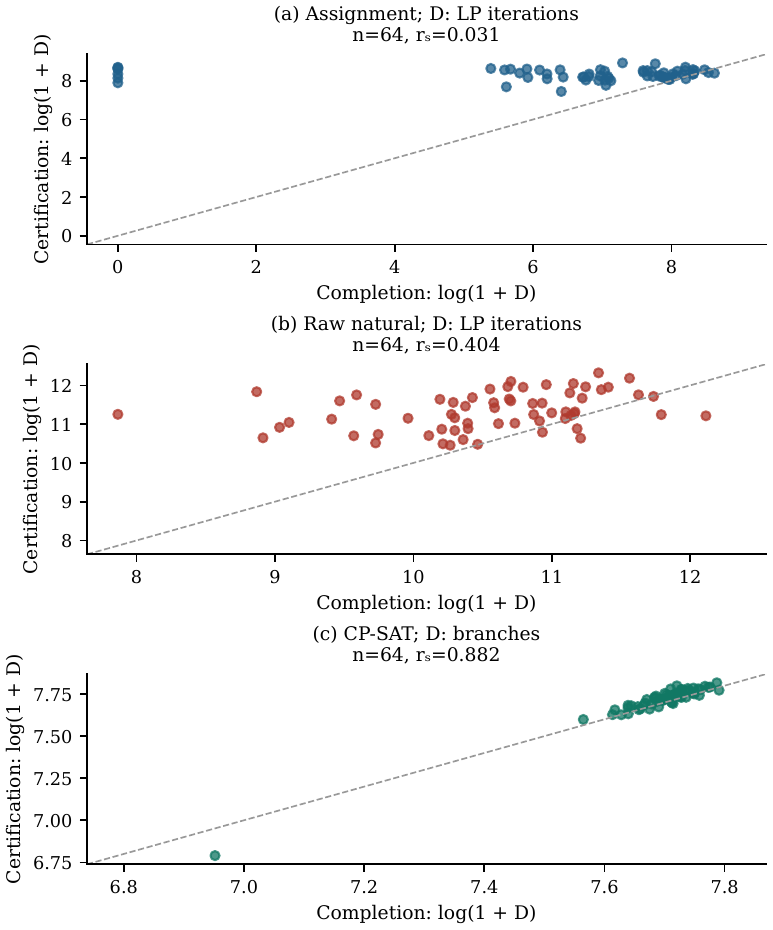}
\caption{Both tasks on every puzzle in the fixed harder sample. Each point is a puzzle; axes show $\log(1+D)$ with the indicated counter. The dashed line marks equal counters within a model. Panels (a), (b) and (c) show assignment, raw natural and CP-SAT, respectively, and have separate scales. Displayed coefficients are ordinary Spearman correlations. These describe task rankings within each model, not equivalence between MILP iterations and CP branches.}\label{fig:hardtasks}
\end{figure}

Paired bootstrap contrasts compare the models on identical resamples. The ordinary natural-minus-assignment correlation difference is \SubmissionNaturalAssignmentOrdinaryDifference\, with interval $[\SubmissionNaturalAssignmentOrdinaryLow,\SubmissionNaturalAssignmentOrdinaryHigh]$. This comparison holds HiGHS and the LP counter fixed. The corresponding within-count difference is \SubmissionNaturalAssignmentWithinDifference\, with interval $[\SubmissionNaturalAssignmentWithinLow,\SubmissionNaturalAssignmentWithinHigh]$, so that adjusted separation is not resolved by this sample. CP-minus-MILP contrasts remain positive in both analyses; their full intervals are in Table~\ref{tab:app-submission_hard_contrasts}. A near-zero rank correlation does not establish statistical independence.

The counter check narrows the interpretation further. Ordinary task correlations using processed MILP nodes are \SubmissionCounterAssignmentR\ for assignment and \SubmissionCounterNaturalR\ for raw natural, broadly resembling their LP-counter results. CP deterministic work instead gives \SubmissionCounterCPR\, with within-count value \SubmissionCounterCPWithin. Thus its particularly strong branch concordance is not a counter-independent property. Table~\ref{tab:app-submission_hard_counters} gives the corresponding intervals.\par

Cross-formulation rankings also vary: assignment LP effort and CP branches have rank correlations \HardCrossFindR\ for completion and \HardCrossCertR\ for certification. The association between the two tasks and the rankings across formulations therefore depend on which puzzle collection is tested.

\subsection{Strengthening and common propagation}
The unit sums revisit a modelling lesson from \cite{williams2026}. Note that requiring sums of forty-five is too weak to enforce all-different on its own, but those equalities can remove fractional possibilities from an already correct integer formulation. Row-rank equalities link cell values more closely to their ordering variables, while explicit singles preprocessing supplies forced values before the solver starts.\par

Table~\ref{tab:variants} summarises their observed effects on the common diagnostic selection. Adding unit sums reduces the number of finished runs with positive LP work, and adding row-rank equalities reduces it further. After applying the same singles preprocessing, the original natural model finishes every diagnostic run, and only a few runs record positive LP counts. These results are based on a mixture of randomly selected and illustrative puzzles, rather than the full benchmark.

\begin{table}[H]\centering\small
\caption{Formulation diagnostics on the same selection and ten-second cap. F denotes completion and C denotes certification. ``Positive'' counts positive LP effort among finished calls; unfinished calls are excluded from that column and shown through the finished denominator. The rank variant includes unit sums. Paired coverage and effort comparisons are in Appendix~\ref{app:task-comparisons}.\vspace{2.0mm}}\label{tab:variants}
\begin{tabular}{lrrrr}
\toprule
Variant & Done F & Positive F & Done C & Positive C \\
\midrule
Assignment & 90/90 & 0 & 90/90 & 0 \\
Raw natural & 88/90 & 24 & 86/90 & 22 \\
All-order & 88/90 & 24 & 87/90 & 23 \\
Row-order & 88/90 & 24 & 87/90 & 23 \\
Unit sums & 89/90 & 15 & 88/90 & 14 \\
Sums + ranks & 88/90 & 7 & 88/90 & 8 \\
Natural + singles & 90/90 & 3 & 90/90 & 2 \\
\bottomrule
\end{tabular}
\par
\end{table}

The shorter order exclusions do not uniformly improve effort. Among puzzles on which both certification solves finish, each sometimes decreases and sometimes increases LP iterations. Their completion models are identical to the baseline and give identical completed LP counts in the observed runs. Table~\ref{tab:pairedcoverage} reports paired coverage alongside completed-pair effort, so a timeout is not treated as a large observed count. Rank strengthening changes the ordering of effort more substantially: its Spearman correlation with the raw model is \RanksFindR\ for completion and \RanksCertR\ for certification on diagnostic puzzles where both solves finish. Ties and the diagnostic time limit remain part of these comparisons. The assignment variants, including the stronger bound of seventy-seven agreeing cells, remain at the zero-LP floor and yield no rank comparison.

The opening example also illustrates why fewer variables alone is an incomplete explanation. For \texttt{SUD-17-0213}, the row-order exclusion has fewer variables but more LP work than the cell-witness exclusion. Unit sums alone do not reduce the LP count to zero, whereas the rank equations and common singles closure do. Separate presolve inspections report reduced nonempty models for the former representations and infeasibility for the latter. These logs identify where the diagnostic operation stops; they do not establish that HiGHS executed our particular singles trace.

The harder sample gives a comparison in which all raw natural and rank-strengthened solves finish with positive LP counts. Their rank correlations are \SubmissionRanksFindR\ for completion and \SubmissionRanksCertR\ for certification. Rank strengthening reduces LP iterations on \SubmissionRanksFindLess\ of \HardN\ completion pairs and \SubmissionRanksCertLess\ of \HardN\ certification pairs, increases them on the remaining pairs, and leaves none equal. The median individual rank/raw ratios are \SubmissionRanksFindRatio\ and \SubmissionRanksCertRatio. Thus it usually reduces this counter while substantially changing the rankings. Table~\ref{tab:app-submission_hard_strengthening} reports the paired comparison; these ratios are not ratios of the marginal medians in Table~\ref{tab:app-review_hard_summary}.

\subsection{What changes when presolve is switched off?}\label{sec:presolve}
Keeping presolve enabled gives a comparison of each formulation as used by the configured solver. It also allows the formulations to benefit differently from the solver's simplifications. To examine that interaction, we repeat both tasks with presolve on and off for both MILP formulations on \PreBenchmarkN\ benchmark puzzles, five at every exact clue count, and the same \PreExternalN\ harder puzzles. All calls are fresh and use the original givens. Neither formulation receives the separate singles preprocessing in this experiment.

The common \PreTimeLimit-second campaign finishes \PrePrimaryFinished\ of \PrePrimaryCalls\ calls. We then rerun all eight conditions on each of the \PreExtendedN\ puzzles with any unfinished call, using a common \PreExtendedLimit-second cap within each puzzle. Every extended call finishes. The comparisons below use these complete matched records; Appendix~\ref{app:presolve} also gives the original fixed-limit coverage and distinguishes the two analyses.

The benchmark's measurement floor persists for assignment, although it is no longer universal: with presolve off, \PreAssignmentFindOffZero\ of \PreBenchmarkN\ completion runs and \PreAssignmentCertOffZero\ certification runs still report zero LP iterations. Raw natural-model zeros fall from \PreNaturalFindOnZero\ to \PreNaturalFindOffZero\ in each task. Thus some zero counts depend on main-model presolve, while others arise without it. The solver can still make deductions during search; setting this option to off does not remove every form of propagation or simplification. LP counts alone cannot measure the total cost of these deductions.

Presolve can also change which instances have large counters. On the harder sample, the on/off completion rank correlations are \PreHardAssignmentFindRank\ for assignment and \PreHardNaturalFindRank\ for the raw natural model. Neither completion ranking is stable across these settings. Switching presolve off sometimes lowers and sometimes raises LP iterations, as Table~\ref{tab:presolve-pairs} shows. For the same \PreCommonPositiveN\ benchmark puzzles with positive counts in both tasks under both settings, completion--certification correlations are \PreCommonOnTaskR\ with presolve on and \PreCommonOffTaskR\ with it off. This supports positive task association in that small common subset, without establishing stability of every instance ranking.

On the harder sample, assignment task association remains weak, changing from \PreHardAssignmentOnR\ to \PreHardAssignmentOffR. The natural-model estimate falls from \PreHardNaturalOnR\ to \PreHardNaturalOffR, with off-setting bootstrap interval $[\PreHardNaturalOffLow,\PreHardNaturalOffHigh]$. These small-sample estimates do not establish a general difference between the two settings. They show why the primary task correlations should be read as properties of the stated solver configuration. The benchmark's contrast in zero-iteration frequencies persists, but numerical rankings and task association remain conditional on the solver settings.

\subsection{Digit labels change natural-model effort}
Relabelling every digit consistently leaves the puzzle unchanged up to the names of its symbols. Clue count, clue-distribution imbalance, candidate sizes and singles deductions are preserved. Yet raw natural-model LP effort varies across labellings on \PermutationFindVariableBases\ of the twelve bases. For \texttt{SUD-17-0041}, all labellings finish and certification ranges from \PermutationCaseMin\ to \PermutationCaseMax\ LP iterations. These preserved features therefore cannot fully determine the solver's effort, even for different labellings of the same puzzle.

\begin{figure}[ht!]\centering\includegraphics[width=0.8\linewidth]{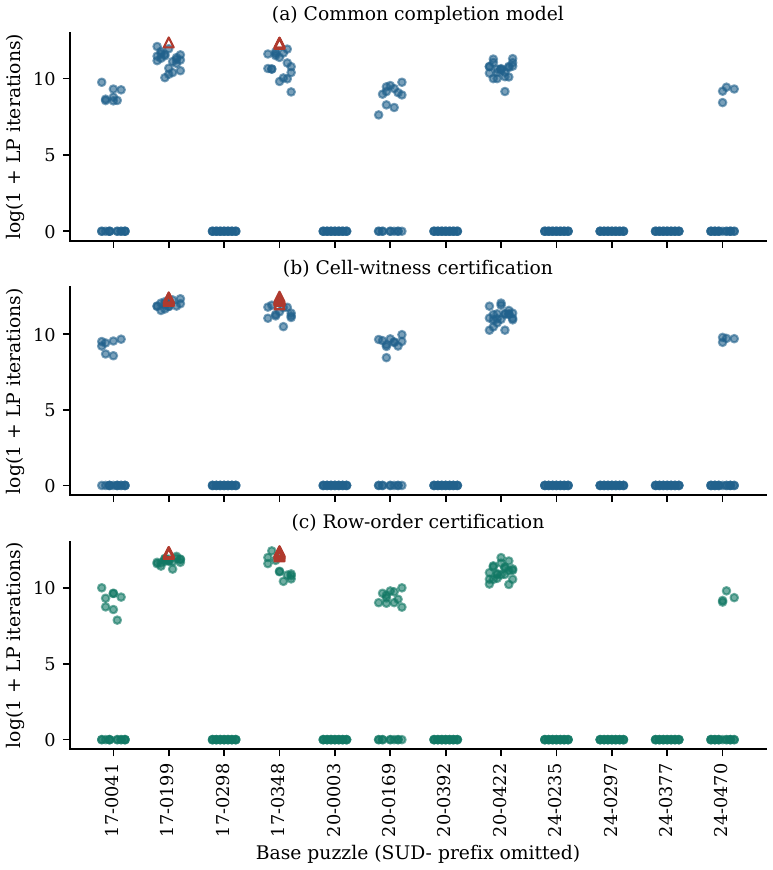}
\caption{Digit-permutation sensitivity with a fixed solver seed. Each horizontal position is one base puzzle; dots are finished calls for the identity and seeded relabellings, with a small deterministic horizontal offset to reveal ties. Open triangles show the work accrued by unfinished calls, not their unknown completed effort. Panel (a) shows their common completion model once; panels (b) and (c) compare the two certification exclusions. The shared horizontal axis identifies the base puzzles. Assignment calls, omitted here, all report zero LP iterations.}\label{fig:permutations}
\end{figure}

Raw natural completion finishes \PermutationFindDone\ of the labelling runs and certification finishes \PermutationCertDone, each from the same twelve groups. Restricting to the \PermutationCertCompleteBases\ fully finished certification groups, mean within-base log-effort variance is \PermutationCertWithinVariance, compared with \PermutationCertBetweenVariance\ for the variance of base means. This is appreciable within-base variation, although the selected bases and censored groups preclude a population variance decomposition. Row-order exclusion retains the sensitivity. The natural formulation treats digits as numerical cell values. A consistent permutation preserves the logical puzzle, but changes numerical differences and orderings in its big-$M$ inequalities. HiGHS can consequently perform different operations on the relabelled model. In the assignment model it instead permutes digit-indexed variables, although solver behaviour can still depend on ordering. Assignment remains at its measurement floor for every tested labelling; that finding cannot establish label invariance on harder populations.

\subsection{Adding correct clues can increase recorded effort}
Adding a correct clue sometimes increases the recorded effort, even though it leaves the puzzle with the same unique completion. This may seem surprising, since we would ordinarily expect an additional clue to make a Sudoku puzzle easier to solve by hand, assuming all existing clues remain unchanged. Natural-model completion LP iterations increase on \NestedFindUp\ of \NestedFindDone\ finished adjacent pairs, involving \NestedFindBases\ base puzzles, while certification increases on \NestedCertUp\ of \NestedCertDone\ pairs, involving \NestedCertBases\ bases. Most pairs tie. Among non-tied pairs, the increase fractions are \NestedFindNonTieFraction\ and \NestedCertNonTieFraction, showing that increases are more common when considering only pairs with a change in effort. Median increases are \NestedFindMedianUp\ and \NestedCertMedianUp\ LP iterations. Appendix~\ref{app:margin-chains} provides unresolved-pair counts and intervals obtained by resampling base puzzles.

\begin{figure}[ht!]\centering\includegraphics[width=0.85\linewidth]{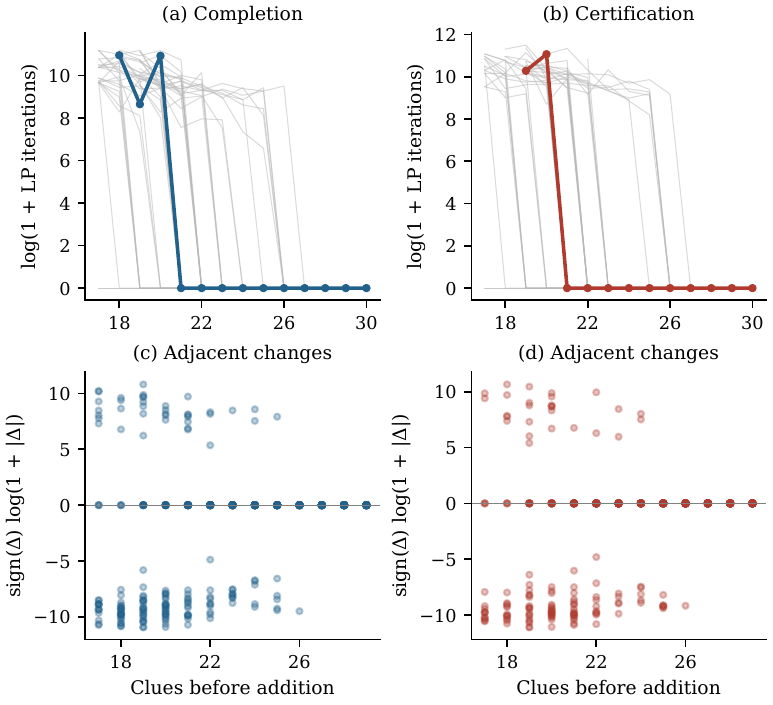}
\caption{Natural-model trajectories and adjacent changes in the active part of the nested sequences. In panels (a) and (b), grey lines follow all seeded addition orders; the darker line with markers highlights base 19, order 1, used in the worked example. Missing observations are unfinished calls. Panels (c) and (d) retain only finished adjacent pairs and use a signed log scale to show both increases and decreases. The full sequences extend to eighty clues; later zero-work steps are retained in the reported denominators.}\label{fig:chains}
\end{figure}

One selected addition raises fresh diagnostic completion effort from \ChainBeforeLP\ to \ChainAfterLP\ LP iterations. The separate presolve model nevertheless shrinks from \ChainBeforeRows\ rows and \ChainBeforeCols\ columns to \ChainAfterRows\ rows and \ChainAfterCols\ columns. Rank strengthening or common singles preprocessing reduces both solves to zero LP iterations. Adding the clue does not change the solution: the same unique completed grid remains feasible before and after the addition. What changes is the formulation presented to the solver. Reduced model dimensions alone do not explain the increase. This is a change in algorithmic effort despite an unchanged integer solution set; it is not a claim about human solving times.

Finally, the chains test the source-generation boundary descriptively. In the benchmark, singles complete \TwentyThreeSolved\ puzzles at twenty-three clues but only \DeletionTwentyFourSolved\ at twenty-four clues, despite the extra given. In the overlapping augmentation sample at twenty-four clues, \AugmentationTwentyFourSolved\ of forty records are singles-complete. These dependent chains are too small to estimate a general generator effect, but the contrast prevents reading the benchmark discontinuity as a pure effect of adding one clue. CP work also shows nonmonotone steps in the chains, whereas assignment remains at its zero-counter floor.

\subsection{Structural prediction: useful associations with limited scope}
Clue count already helps predict whether a natural-model solve records any LP iterations; candidate and propagation features improve the held-out scores. Predicting the size of a positive count is harder, although these features also provide gains. For positive natural-model completion effort, held-out $R^2$ rises from \NaturalPositiveFindBaseR\ to \NaturalPositiveFindFullR; for certification it rises from \NaturalPositiveCertBaseR\ to \NaturalPositiveCertFullR. The corresponding gains are \NaturalPositiveFindGain\ and \NaturalPositiveCertGain, with evaluation-bootstrap intervals $[\NaturalPositiveFindGainLow,\NaturalPositiveFindGainHigh]$ and $[\NaturalPositiveCertGainLow,\NaturalPositiveCertGainHigh]$. MAE also improves on these positive populations. They contain \PositiveModelFindN\ and \PositiveModelCertN\ eligible observations after the minimum stratum-size rule.

On the full benchmark, however, $R^2$ rises only from \NaturalAllFindBaseR\ to \NaturalAllFindFullR\ for completion and from \NaturalAllCertBaseR\ to \NaturalAllCertFullR\ for certification. MAE worsens from \NaturalAllFindBaseMAE\ to \NaturalAllFindFullMAE\ and from \NaturalAllCertBaseMAE\ to \NaturalAllCertFullMAE, respectively. A small improvement in squared error is therefore not a uniform predictive improvement. Most clue-count groups contain mainly zero counts. A shared coefficient for each added feature changes the errors across these groups, improving some predictions while worsening others. Separate grid and digit ablations each show small gains. Their combined effect should not be attributed to clue placement alone. In the lower generator range, the full feature block raises $R^2$ from \NaturalLowFindBaseR\ to \NaturalLowFindFullR\ for completion and from \NaturalLowCertBaseR\ to \NaturalLowCertFullR\ for certification. These restricted fits reinforce the need to state the population.

This MAE deterioration is not solely caused by negative predictions. The full model produces \ReviewFindNegativeN\ negative completion predictions and \ReviewCertNegativeN\ negative certification predictions. Clipping them at zero lowers MAE to \ReviewFindClippedMAE\ and \ReviewCertClippedMAE, respectively, still above their count-only values. Table~\ref{tab:clipping} reports prediction minima and both clipped and unclipped scores. Positive-only conclusions are assessed on their own eligible populations.

\begin{figure}[ht!]\centering\includegraphics[width=0.85\linewidth]{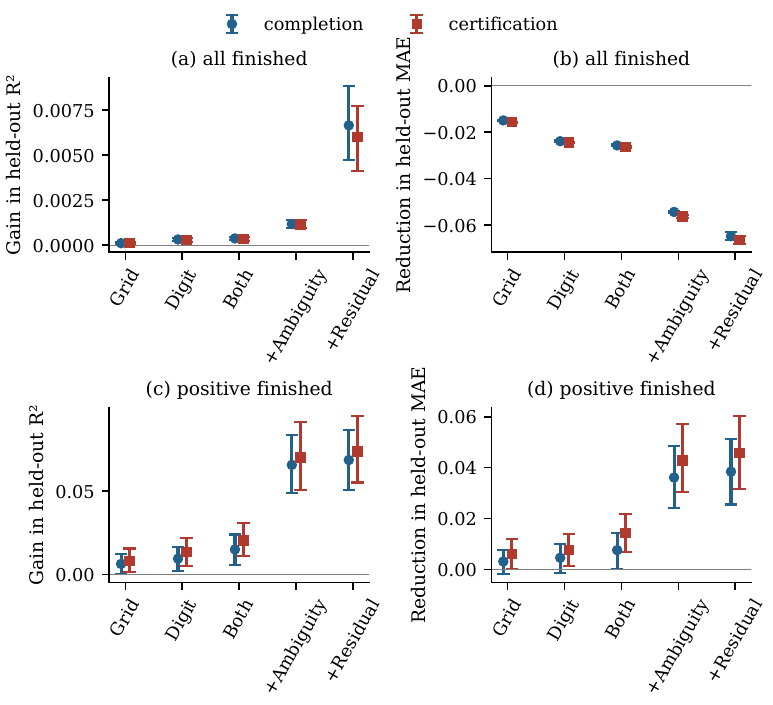}
\caption{Paired predictive gains from the original, unclipped predictions over exact-count intercepts for natural-model $\log(1+\mathrm{LP})$. Grid-only and digit-only ablations are followed by their combination, candidate ambiguity and singles residual count. Positive values favour the added features. Intervals resample held-out records within count, conditional on the fitted folds. Panels (a) and (b) use all finished calls; panels (c) and (d) use positive finished calls. Circles denote completion and squares certification. The two populations answer different prediction questions.}\label{fig:prediction}
\end{figure}

For CP deterministic work, the strongest observed predictor in the tested compact feature set is the number of cells unresolved after singles propagation. Simply recording whether any cells remain unresolved raises held-out $R^2$ to \CPAllFindIndicatorR\ for completion and \CPAllCertIndicatorR\ for certification. Using the number of unresolved cells raises these to \CPAllFindFullR\ and \CPAllCertFullR. Among puzzles with cells still unresolved, count-only values are \CPResidualFindBaseR\ and \CPResidualCertBaseR, rising to \CPResidualFindFullR\ and \CPResidualCertFullR\ with the full block. Thus the number of unresolved cells supplies information beyond a simple split between singles-complete and other puzzles. Table~\ref{tab:app-revision_cv_cp} gives the scores for the candidate, residual-indicator and residual-count models.

The cost comparison limits the practical interpretation. Across repeated serial batches on the same \SerialTimingN\ benchmark puzzles from the diagnostic selection (excluding the nested pair), full feature extraction takes \FeatureBatchMin--\FeatureBatchMax\ seconds, compared with \CPFindBatchMin--\CPFindBatchMax\ for CP completion and \CPCertBatchMin--\CPCertBatchMax\ for CP certification, including model construction. Singles alone take \SinglesBatchMin--\SinglesBatchMax\ seconds and are already included in the full-feature time. The features help explain recorded work, but their computation is not demonstrated to be a cheaper screening device than solving these standard-size instances directly.

The number of unresolved cells is less informative on the harder sample. No initial single is available in any harder-sample puzzle, so closure fills no blank and $U=81-C$ throughout that sample. At each exact count, $U$ is constant and adds no information beyond clue count for distinguishing those puzzles. This is a deduction from the observed propagation profile, not an estimate of external prediction error. Candidate and distribution features may still vary; we have not evaluated transfer of the full fitted predictor.

\subsection{An exploratory clue-deletion margin result}
All \SubmissionMarginPositiveN\ sampled puzzles with positive natural-model certification LP counts have a minimum clue-deletion margin of one. Recall that this means there is a single clue whose removal allows another completion. However, among puzzles with the same number of clues, the margin varies only in groups where every puzzle records zero LP iterations. Consequently, this sample cannot establish whether the margin helps explain differences in positive computational effort beyond the number of clues.

Including the margin in the predictive model also provides no improvement for CP-SAT deterministic work. These findings do not rule out a relationship between the margin and computational effort in other puzzle collections. Appendix~\ref{app:margin-results} provides the sample composition and predictive comparisons, while Appendix~\ref{app:validation} reports the independent optimality checks.


\section{Discussion}
Our first finding is that the mathematical formulation affects the information available to the solver. The raw natural formulation correctly enforces the Sudoku rules for integer solutions, but its LP relaxation imposes no restrictions on the cell values beyond their bounds and the original clues. By contrast, when a puzzle can be completed using only naked and hidden singles, the assignment LP relaxation determines the unique completion. Adding redundant equalities, rank constraints or explicit singles preprocessing can strengthen the natural formulation, although their effects on computational effort depend on the puzzle and solver.

Our second finding concerns the relationship between completion and uniqueness certification. Completion requires finding one valid grid, whereas fresh certification must rule out every alternative to a supplied completion. These tasks show a strong association under the raw natural formulation on the main benchmark, but the relationship changes on the harder puzzles and depends on how computational effort is measured. In particular, the weaker association for assignment LP iterations and stronger association for CP-SAT branches are not universal properties of these approaches. Moreover, when singles complete a puzzle, the same deductions already establish uniqueness without requiring a separate search.

A third finding is that changes to the mathematical formulation or puzzle representation can affect computational effort even when the underlying solution remains unchanged. Changing presolve settings or consistently relabelling digits can alter solver behaviour. Similarly, adding a correct clue preserves the unique completion but can sometimes increase the number of LP iterations. Thus, making a puzzle more constrained does not necessarily reduce the computational effort required by a particular solver.

Finally, structural features help explain some of the variation in computational effort among puzzles with the same number of clues. Candidate ambiguity and the number of cells unresolved after singles provide useful information beyond clue count, although their predictive value depends on the puzzle collection and effort measure. Taken together, these findings suggest that computational difficulty should be understood in relation to the puzzle, mathematical formulation, task, solver configuration and measure of effort, rather than treated as a fixed property of the puzzle itself.


\section{Limitations}
The main benchmark \cite{dataset2026} is balanced by exact clue count, rather than sampled uniformly from all uniquely solvable Sudoku puzzles. Its two generation procedures cover different clue-count ranges, making it difficult to separate the effects of clue count and puzzle generation. The external collections broaden the comparison but are not representative of all Sudoku puzzles, and none of the collections measures human solving difficulty.

Although all main-benchmark runs finish after staged recovery, some diagnostic and relabelling experiments contain unfinished runs. Comparisons based only on completed runs may therefore be affected by which puzzles finish within the time limit. Our primary measures are solver operation counts rather than elapsed time, and the results remain dependent on the solver settings and computational environment.

The experiments use one MILP solver and one CP-SAT implementation, so the findings do not establish a universal ranking of formulations or solvers. Some additional experiments, including digit relabelling, clue additions and clue-deletion margins, use relatively small samples. The predictive results describe associations within the tested collections and do not necessarily extend to other puzzle generators or unseen clue counts. In particular, the evidence is insufficient to establish whether the clue-deletion margin predicts computational effort beyond clue count.

Finally, the computational uniqueness checks rely partly on floating-point solver results rather than independently verifiable formal certificates. Additional integer and CP checks provide supporting evidence, while the mathematical proofs establish the correctness of the formulations.


\section{Conclusion}
Finding a Sudoku completion and proving that it is unique do not necessarily require similar computational effort. Our experiments show that the relationship between these tasks depends on the mathematical formulation, solver and measure of effort. On the main benchmark, completion and certification effort are strongly associated under the raw natural formulation. However, on the harder puzzles, this relationship is weak for the assignment formulation, moderate for the natural formulation and strong for CP-SAT when measured by branching. These comparisons also depend on the effort measure, with weaker associations observed for CP-SAT under deterministic work.

Our mathematical analysis helps explain these differences. Although the integer formulations describe the same completed Sudoku grids, their LP relaxations can behave very differently. The experiments also show that strengthening the formulation, changing presolve settings, relabelling digits or adding correct clues can change the solver's recorded effort, even when the underlying completed grid remains unchanged. These findings highlight why computational difficulty should be considered alongside the formulation, solver configuration, task and measure of effort, rather than as a property of the puzzle alone.

Future work could extend these comparisons to generalised $n^2\times n^2$ Sudoku, allowing us to study how computational effort changes with grid size. Further mathematical questions include identifying when the assignment LP relaxation determines a unique solution beyond the sufficient singles condition, and understanding how additional equalities and rank constraints strengthen the natural formulation.


\section*{Benchmark Data Availability}
The Sudoku benchmark \citep{dataset2026} used in this study, comprising \NInstances\ uniquely solvable instances and their recorded completions and uniqueness-verification results, is publicly available at \url{https://doi.org/10.5281/zenodo.22959803}. The additional computational experiments and analyses reported here were conducted separately, as described in Sections~\ref{sec:method}--\ref{sec:results} and the appendices. 



\clearpage


\section*{Appendix}
\appendix
\numberwithin{equation}{section}
\numberwithin{figure}{section}
\numberwithin{table}{section}
\numberwithin{proposition}{section}
\renewcommand{\theHfigure}{\thesection.\arabic{figure}}
\renewcommand{\theHtable}{\thesection.\arabic{table}}


\section{Additional Mathematical Details}\label{app:math}
\subsection{Projection calculations}\label{app:math-projection}

For Proposition~\mainref{prop:box}, let $\Delta=z_u-z_v\in[-8,8]$ and $y=(8-\Delta)/16$. The slacks in the two pair inequalities are respectively $7(8+\Delta)/16$ and $7(8-\Delta)/16$, so both are nonnegative. If a consecutive reference pair is stored in reverse cell order, its disagreement term is $y$ rather than $1-y$; either way it equals $(8+z_{v_k}-z_{v_{k+1}})/16$. This justifies the telescoping argument without changing the fixed orientation of any stored variable.

The completion projection argument extends to the digit domain $\{1,\ldots,q\}$ for every $q\ge2$, using the tight raw pairwise constant $M=q$. For $\Delta\in[-(q-1),q-1]$, the interval
\[
 [(1-\Delta)/q,(q-1-\Delta)/q]\cap[0,1]
\]
is nonempty: its endpoints are ordered, its lower endpoint is at most one and its upper endpoint is at least zero. This general observation concerns pairwise completion relaxations; the certification construction in the article is stated for standard Sudoku.

\subsection{Extrema of clue imbalance}\label{app:math-imbalance}
For integer counts $v_1,\ldots,v_9\in[0,9]$ summing to $C=9q+s$, $0\le s<9$,
\[
 I_{\max}(C)=\max\sum_a(v_a-C/9)^2=81q+s^2-C^2/9.
\]
\begin{proof}
The objective differs from $\sum_a v_a^2$ by the fixed quantity $C^2/9$. If $0<v_a\le v_b<9$, transferring one unit from $v_a$ to $v_b$ increases the sum of squares by $2(v_b-v_a)+2>0$. A maximiser therefore has at most one entry strictly between zero and nine. Its nonzero entries are $q$ copies of nine and, when $s>0$, one copy of $s$, giving the formula.
\end{proof}
The corresponding minimum is
\[
 I_{\min}(C)=s(9-s)/9.
\]
Indeed, transferring one unit from a count at least two larger than another strictly reduces the sum of squares; a minimiser therefore has counts differing by at most one. Substituting $s$ copies of $q+1$ and $9-s$ copies of $q$ gives the formula.

\subsection{Unavoidable sets and the clue-deletion margin}\label{app:math-margin}
The characterisation in Proposition~\ref{prop:margin} follows by comparing difference supports with minimal unavoidable sets.
\begin{proof}[Proof of Proposition~\ref{prop:margin}]
For an alternative completed grid $z$, the difference support $U(z)=\{(i,j):z_{ij}\ne\bar z_{ij}\}$ is an unavoidable set: fixing its complement leaves at least two completions. Every such finite set contains a minimal unavoidable set. Conversely, a minimal unavoidable set admits a competing grid whose difference support is exactly that set; a strictly smaller support would contradict minimality. Minimising $|\C(P)\cap U|$ over these sets therefore gives the same minimum as \eqref{eq:rho}.
\end{proof}
For any integer $r\ge1$, the intersection formula gives $\rho(P)\ge r$ if and only if every minimal unavoidable set meets the clues in at least $r$ cells. Equivalently, deleting any set of at most $r-1$ clues preserves uniqueness.

There is also a precise criterion for a new clue to increase the margin. Let $\mathcal W(P)$ contain the difference supports of alternative grids attaining $\rho(P)$, and add a correct clue at a previously blank cell $v$. The margin rises by one if and only if $v$ belongs to every member of $\mathcal W(P)$. If an optimal alternative agrees at $v$, its old objective remains attainable. Otherwise every old optimum gains one, while every nonoptimal integer objective was already at least one larger. Corollary~\ref{cor:rho} then forces equality. This criterion concerns robustness, not a tested policy for reducing solver effort.

The quantity $\sum_{i,j}|z_{ij}-\bar z_{ij}|$ instead measures an $L_1$ distance between arbitrary digit labels. It is not Hamming distance and need not survive relabelling: changing $1$ to $9$ contributes eight, whereas changing $4$ to $5$ contributes one. Cell disagreement is represented by $H(z,\bar z)=\sum_{i,j}\ind\{z_{ij}\ne\bar z_{ij}\}$; the margin counts only disagreements at original clue positions.


\section{Computational Methodology and Validation}\label{app:implementation}
\subsection{Source records and independent checks}\label{app:data}\label{app:validation}
Every benchmark puzzle and stored solution was reconstructed from the clue-count archives, checking cell order, exact clue counts, clues and Sudoku units. All archived uniqueness records were inspected. They contain optimal objectives and bounds for the earlier $-81$ test, independent solution counts, source code and logs. These original verification records are distinct from the solver measurements in this paper. The Zenodo API checksums pass; one stale README entry in the archive's internal SHA-256 manifest is a documented source exception. The original files are unchanged, and their concurrent timings are not used as new experimental measurements.

The archived construction code uses uncapped integer search to obtain completed grids and test clue deletions, rejecting attempts that cannot reach the target count. Its later AMPL verification uses a 120-second cap and rejects missing results, nonoptimal status, inconsistent completions or unmatched uniqueness bounds. It does not substitute another puzzle after a verification timeout. The released records all pass that verification; discarded construction attempts and decisions outside the recorded workflow cannot be reconstructed.

Table~\ref{tab:validation} gives the independent checks. The benchmark sample takes five puzzles per count with seed 20260926. The twenty nested bases use seed 20261226 and share \ValidationBaseOverlap\ records with that sample. Integer enumeration exhausts the completion space or stops at a second solution; uniqueness requires exhaustion with exactly one completion. The margin check takes the lexicographically first selected ID at each count and uses independent CP optimisation with a five-second cap, seed zero and no imported MILP lower bound. All \IndependentMarginN\ checks return matching optimal objectives and bounds, covering margins one through four with frequencies \MarginCrossOne, \MarginCrossTwo, \MarginCrossThree\ and \MarginCrossFour.

\begin{table}[H]\centering\small
\caption{Independent validation coverage. Resolved means exhaustive uniqueness enumeration or matching optimal bounds, as appropriate. Checks of the archived verification records cover the full benchmark separately.\vspace{2.0mm}}\label{tab:validation}
\begin{adjustbox}{max width=\linewidth}
\begin{tabular}{lrrrr}
\toprule
Check & Submitted & Resolved & Disagree & Unresolved \\
\midrule
Benchmark: integer backtracking & 320 & 320 & 0 & 0 \\
Nested bases: integer backtracking & 20 & 20 & 0 & 0 \\
Norvig: before deduplication & 106 & 106 & 0 & 0 \\
Norvig: retained & 105 & 105 & 0 & 0 \\
Margin: independent CP & 64 & 64 & 0 & 0 \\
Harder external: integer backtracking & 64 & 64 & 0 & 0 \\
\bottomrule
\end{tabular}
\end{adjustbox}\end{table}

For the external collections, the Norvig lists are checked before symmetry deduplication. The harder sample uses Tdoku's \path{puzzles5_forum_hardest_1905_11+} file, attributed to the Enjoy Sudoku Players Forum's May 2019 collection \citep{tdoku}. Sampling is without replacement, using Python's \texttt{Random} generator with seed 20261007 and its \texttt{sample} method, and preserves selected source-line order. The integer bit-mask enumerator is separate from both MILP and CP-SAT. Every selected puzzle is retained, independently verified and checked for symmetry overlap with the earlier populations.

Finite formulation checks cover local disequalities, alternative-solution fixtures, ordering exclusions, rank equations, fractional projection witnesses, imbalance extrema and symmetry transformations. Returned grids are checked against all clues and units, and against the stored reference where uniqueness requires equality. Certification outputs are checked for reference exclusion; timeouts never establish uniqueness. The exported numerical observations were reconciled with their saved solver records, preserving the distinction between missing values and zero. These checks support the computations without turning floating-point infeasibility reports into portable formal proofs.

\subsection{Exact symmetry checks}\label{app:symmetry}
The coloured incidence graph contains all 81 cells, 18 row or column vertices sharing one colour, nine block vertices and nine digit vertices sharing another colour. Each cell is joined to its row, column and block; a given cell is also joined to its digit. Thus the graph has 117 vertices and $243+C$ edges. Count signatures and colour-refinement hashes reject impossible matches before exact colour-preserving VF2 isomorphism (NetworkX 3.2.1). Hash equality alone is never accepted as equivalence. No equivalent pair is found among the \NInstances\ benchmark records.

This graph represents precisely the tested Sudoku symmetries. Its line-intersection graph is $K_{9,9}$, so an isomorphism preserves the row and column families or swaps them globally. Block incidence restricts row and column permutations to bands and stacks and permutations within them; digit vertices permit consistent relabelling. Conversely, each such Sudoku symmetry induces a coloured graph isomorphism. The absence of equivalent benchmark pairs removes the need to group those pairs in cross-validation, but does not exclude other dependence arising from generation.

\subsection{Solver configuration and recorded effort}\label{app:solver}
Cells use row-major ordering, with assignment variables ordered by cell and then increasing digit. Peer pairs are deduplicated and sorted; natural cell variables precede ordering binaries and then disagreement witnesses. For $u<v$, the convention is $y_{uv}=1$ when $z_u<z_v$. Assignment clues are imposed by variable lower bounds. Unit sums and ranks are added as specified in Section~\ref{sec:relaxation}. Common-singles variants pass the same closure as bound fixings to the original models, recording preprocessing separately and retaining the original clue count for analysis.

HiGHS 1.15.1 uses one thread, parallel mode off, seed zero, default heuristics, feasibility tolerances $10^{-7}$ and zero absolute and relative MIP gaps. Presolve is on except in the named sensitivity calls. OR-Tools 9.15.6755 uses one search worker, seed zero and first-solution stopping. Completion receives no reference hint; certification starts afresh with the reference excluded. Separate worker pools avoid interactions between the packages' bundled libraries. The initial natural screen uses six workers; the main assignment/CP campaign, recovery and diagnostics use four single-threaded workers. Other application activity is uncontrolled.

MILP LP iterations come from the full solve report, rather than just the final simplex-information field. Missing early-return fields remain missing. CP deterministic time is a solver-specific work proxy. A standalone presolve inspection uses a separate model copy after each diagnostic solve, with the same thread count, seed and cap but otherwise library defaults, including MIP feasibility tolerance $10^{-6}$. Its reduced dimensions do not reconstruct the measured search state. Zero dimensions require a status: they can mean an empty feasible model or detected infeasibility. Neither these inspections nor CP probing logs identify a particular human deduction sequence.

The serial timing comparison uses the same \SerialTimingN\ benchmark puzzles from the diagnostic selection, excluding the nested pair, with three serial repeats and operation order reversed in the middle repeat. CP durations include model construction; full features include singles closure. These batches support the local ranges reported in Section~\ref{sec:results}, rather than precise millisecond rankings.

\subsection{Exact singles procedure and feature conventions}\label{app:features}
Clue counts and distribution features refer to original givens. Population variances divide by nine and entropies use natural logarithms. Initial candidates eliminate given peer digits once. Initial forced-cell counts count distinct cells, avoiding repeated counts of the same inference. The deterministic closure uses synchronous rounds:
\begin{enumerate}[nosep]
\item Recompute all empty-cell candidate sets from the current grid.
\item Collect the assignments given by naked singles.
\item For each unit and absent digit, collect its assignment if exactly one cell admits that digit.
\item Check that all proposals agree, apply them together, and count one successful round.
\item Stop when no proposal remains; the final unsuccessful inspection is not a round.
\end{enumerate}
Newly solved cells exclude givens, and solved proportions divide by the original blank count. Residual edges join unresolved peer cells once; empty-graph summaries are zero. The normalised imbalance indices use the extrema in Appendix~\ref{app:math-imbalance}. A degenerate index is set to zero and flagged; within this benchmark that occurs at eighty clues. No additional Sudoku technique is included in this closure.

\subsection{Statistical specifications}\label{app:prediction}
Regression uses unpenalised exact-count intercepts: training-stratum means are removed from predictors and outcome, ridge is fitted to the residuals, and the means are restored for prediction. Numerical coefficients are standardised within each training fold, with ridge penalty one. The logistic occurrence model penalises all non-intercept coefficients, including count indicators, using inverse penalty one. There is no held-out tuning.

Five fixed folds are stratified by exact count. Regression populations retain only counts with at least five eligible observations, and all feature blocks use identical records and folds within a population. Scores pool out-of-fold predictions rather than averaging fold-specific ratios. Paired predictive intervals use five hundred seeded resamples within count, applying the same indices to baseline and extended predictions. They describe evaluation-loss variation conditional on the fitted folds, not a refitting bootstrap or transfer to a new generator. Clipping predictions at zero uses the same fitted values and does not increase pointwise absolute or squared error for nonnegative outcomes.

For task association, we first select the stated positive or full population. For the within-count statistic, we calculate midranks separately within each exact clue count, subtract the mean rank in each stratum and correlate the pooled centred vectors. Constant strata contribute zero variation; this is not an equally weighted average of stratum correlations. Five hundred within-count paired resamples recompute the ranks, using seed 20261004 for the benchmark and 20261007 for the harder sample. Paired model contrasts use the same resampled puzzle IDs and subtract the two coefficients. These are descriptive percentile intervals, not multiplicity-adjusted hypothesis tests. The distinct presolve bootstrap is specified in Appendix~\ref{app:presolve}.

Relabelling variances use $\log(1+\mathrm{LP})$ and only bases whose complete identity-plus-permutation group finishes. The variance of base means includes finite-repetition noise and is not an unbiased random-effects component. Clue-frequency summaries, candidate cardinalities and singles residual structure are verified invariant under these permutations. Nested-chain intervals resample whole base puzzles, retaining both addition orders together. An all-zero empirical increase rate is not a population guarantee of monotonicity.


\section{Additional Experimental Results}\label{app:comparisons}
\subsection{Natural-model recovery}\label{app:recovery}
Only unfinished calls from the one-second screen are restarted, with limits of ten, sixty and three hundred seconds. Each restart is a fresh solve with the same formulation and settings. The retained counter is from the first successful attempt; unsuccessful attempts are not added to it. Table~\ref{tab:app-revision_recovery} gives the stage coverage. All \NInstances\ puzzles ultimately have both tasks finished.

\begin{table}[H]\centering\small
\caption{Natural-model recovery stages. Finished counts refer to calls at that stage, not cumulative coverage. Unfinished counters are never substituted for completed effort.\vspace{2.0mm}}\label{tab:app-revision_recovery}
\begin{adjustbox}{max width=\linewidth}
\begin{tabular}{llrrr}
\toprule
Stage & Task & Calls & Finished & Unresolved \\
\midrule
bigm\_full & certification & 32,000 & 30,555 & 1,445 \\
bigm\_full & completion & 32,000 & 30,869 & 1,131 \\
recovery10 & certification & 1,445 & 1,207 & 238 \\
recovery10 & completion & 1,131 & 1,029 & 102 \\
recovery60 & certification & 238 & 214 & 24 \\
recovery60 & completion & 102 & 93 & 9 \\
recovery300 & certification & 24 & 24 & 0 \\
recovery300 & completion & 9 & 9 & 0 \\
\bottomrule
\end{tabular}
\end{adjustbox}\end{table}

In the saved one- versus five-second comparison, each task has 640 matched calls. Both finish on 618 completion pairs and 608 certification pairs; LP and node counts agree exactly on every such pair. This does not establish invariance for all time limits or loads. Including work from abandoned attempts gives the task associations in Table~\ref{tab:app-review_workflow}, close to the retained-call values in this experiment.

\begin{table}[H]\centering\small
\caption{Retained-call and cumulative recovery-work task associations. All-attempt LP work includes abandoned time-limited calls; neither measure is elapsed time.\vspace{2.0mm}}\label{tab:app-review_workflow}
\begin{adjustbox}{max width=\linewidth}
\begin{tabular}{llrrr}
\toprule
Effort scope & Pairs & N & $r_S$ & $r_{\mathrm{within}}$ \\
\midrule
Retained call & All & 32,000 & 0.991 & 0.922 \\
Retained call & Both positive & 1,928 & 0.841 & 0.781 \\
All attempts & All & 32,000 & 0.991 & 0.922 \\
All attempts & Both positive & 1,928 & 0.840 & 0.780 \\
\bottomrule
\end{tabular}
\end{adjustbox}\end{table}

\subsection{Benchmark task associations}\label{app:benchmark-tasks}
Tables~\ref{tab:conditional} and \ref{tab:conditionalregimes} separate shared-zero effects from within-count association and show the smaller positive population in the upper generation range. Table~\ref{tab:app-review_paired_cost} distinguishes rank association from the relative cost of the two tasks.

\begin{table}[H]\centering\small
\caption{Raw-natural task association, controlling shared zeros and exact clue count. Positive subsets use the stated counter. Intervals apply to the centred within-count statistic; ordinary correlations are descriptive.\vspace{2.0mm}}\label{tab:conditional}
\begin{adjustbox}{max width=\linewidth}
\begin{tabular}{llrrrr}
\toprule
Counter & Pairs & N & $r_S$ & $r_{\mathrm{within}}$ & 95\% interval \\
\midrule
LP iterations & All & 32,000 & 0.991 & 0.922 & [0.911, 0.931] \\
LP iterations & Either positive & 1,962 & 0.846 & 0.785 & [0.757, 0.808] \\
LP iterations & Both positive & 1,928 & 0.841 & 0.781 & [0.751, 0.800] \\
Nodes & All & 32,000 & 0.967 & 0.891 & [0.878, 0.901] \\
Nodes & Either positive & 1,933 & 0.813 & 0.757 & [0.729, 0.780] \\
Nodes & Both positive & 1,807 & 0.813 & 0.748 & [0.718, 0.776] \\
\bottomrule
\end{tabular}
\end{adjustbox}\end{table}

\begin{table}[H]\centering\small
\caption{Within-count association among pairs positive in both tasks, separated by generation range. The sparse upper range gives weaker precision and does not identify a generator effect.\vspace{2.0mm}}\label{tab:conditionalregimes}
\begin{adjustbox}{max width=\linewidth}
\begin{tabular}{llrrr}
\toprule
Counter & Clue range & N & $r_{\mathrm{within}}$ & 95\% interval \\
\midrule
LP iterations & 17--23 & 1,852 & 0.781 & [0.753, 0.802] \\
LP iterations & 24--80 & 76 & 0.486 & [0.238, 0.682] \\
Nodes & 17--23 & 1,743 & 0.748 & [0.717, 0.775] \\
Nodes & 24--80 & 64 & 0.457 & [0.209, 0.685] \\
\bottomrule
\end{tabular}
\end{adjustbox}\end{table}

\begin{table}[H]\centering\small
\caption{Paired raw-natural LP costs. $\Delta=\log(1+D_{\mathrm{cert}})-\log(1+D_{\mathrm{find}})$. Less and more compare certification with completion on the same puzzle.\vspace{2.0mm}}\label{tab:app-review_paired_cost}
\begin{adjustbox}{max width=\linewidth}
\begin{tabular}{lrrrrrr}
\toprule
Pairs & N & Cert. less & Equal & Cert. more & Median $\Delta$ & IQR of $\Delta$ \\
\midrule
All & 32,000 & 389 & 30,038 & 1,573 & 0.000 & [0.000, 0.000] \\
Either positive & 1,962 & 389 & 0 & 1,573 & 0.490 & [0.095, 0.975] \\
Both positive & 1,928 & 362 & 0 & 1,566 & 0.497 & [0.119, 0.977] \\
\bottomrule
\end{tabular}
\end{adjustbox}\end{table}

\subsection{Harder-sample comparisons}\label{app:external}
The principal correlations appear in Table~\ref{tab:hard}. Tables~\ref{tab:app-review_hard_summary}--\ref{tab:app-submission_hard_counters} give their coverage, count composition, paired contrasts and alternative-counter checks. The one-puzzle count stratum contributes no centred rank variation. Within-count resampling preserves the observed frequencies and does not account for a different source-count mixture.

\begin{table}[H]\centering\small
\caption{Harder-sample coverage and effort under the common ten-second cap. F denotes completion and C certification. Medians use finished calls; LP iterations and branches are different counters.\vspace{2.0mm}}\label{tab:app-review_hard_summary}
\begin{adjustbox}{max width=\linewidth}
\begin{tabular}{llrrrr}
\toprule
Model & Counter & F finish & C finish & F median & C median \\
\midrule
Assignment & LP & 64 & 64 & 1,223.5 & 4,188.5 \\
CP-SAT & Branches & 64 & 64 & 2,231.0 & 2,278.0 \\
Raw natural & LP & 64 & 64 & 42,147.0 & 79,259.5 \\
Natural + singles & LP & 64 & 64 & 42,147.0 & 79,259.5 \\
Sums + ranks & LP & 64 & 64 & 27,846.0 & 50,657.5 \\
Row-order & LP & 64 & 63 & 42,147.0 & 86,347.0 \\
\bottomrule
\end{tabular}
\end{adjustbox}\end{table}

\begin{table}[H]\centering\small
\caption{Harder-sample count frequencies and positive assignment-completion LP counts. Every assignment task finishes.\vspace{2.0mm}}\label{tab:hardcounts}
\begin{adjustbox}{max width=\linewidth}
\begin{tabular}{rrr}
\toprule
Clues & Puzzles & Positive assignment completion LP \\
\midrule
21 & 1 & 1 \\
22 & 8 & 8 \\
23 & 20 & 16 \\
24 & 21 & 21 \\
25 & 9 & 7 \\
26 & 5 & 5 \\
\bottomrule
\end{tabular}
\end{adjustbox}\end{table}

\begin{table}[H]\centering\small
\caption{Paired differences in harder-sample task correlation (left minus right), on all \HardN\ pairs. Each of 500 resamples uses identical puzzle IDs for both models; every coefficient is defined. Intervals describe this curated population.\vspace{2.0mm}}\label{tab:app-submission_hard_contrasts}
\begin{adjustbox}{max width=\linewidth}
\begin{tabular}{lllrr}
\toprule
Left & Right & Statistic & Difference & 95\% interval \\
\midrule
Raw natural & Assignment & Ordinary & 0.373 & [0.056, 0.692] \\
Raw natural & Assignment & Within count & 0.371 & [-0.069, 0.798] \\
CP-SAT & Assignment & Ordinary & 0.851 & [0.583, 1.140] \\
CP-SAT & Assignment & Within count & 0.938 & [0.551, 1.229] \\
CP-SAT & Raw natural & Ordinary & 0.478 & [0.260, 0.699] \\
CP-SAT & Raw natural & Within count & 0.566 & [0.210, 0.887] \\
\bottomrule
\end{tabular}
\end{adjustbox}\end{table}

\begin{table}[H]\centering\footnotesize
\caption{Alternative task counters on all \HardN\ harder-sample pairs. MILP nodes include the root. CP deterministic work is solver-specific; intervals use count-stratified resampling.\vspace{2.0mm}}\label{tab:app-submission_hard_counters}
\begin{adjustbox}{max width=\linewidth}
\begin{tabular}{llrrrr}
\toprule
Model & Counter & $r_S$ & 95\% interval & $r_{\mathrm{within}}$ & 95\% interval \\
\midrule
Assignment & Nodes & 0.048 & [-0.183, 0.264] & -0.047 & [-0.358, 0.261] \\
Raw natural & Nodes & 0.394 & [0.162, 0.586] & 0.182 & [-0.187, 0.496] \\
CP-SAT & Det. work & 0.496 & [0.298, 0.659] & 0.441 & [0.151, 0.643] \\
\bottomrule
\end{tabular}
\end{adjustbox}\end{table}

\begin{table}[H]\centering\footnotesize
\caption{Additional harder-sample task correlations using LP iterations. Assignment here is restricted to pairs positive in both tasks; the other rows use every jointly finished pair. The raw and CP comparisons are in Table~\ref{tab:hard}.\vspace{2.0mm}}\label{tab:app-review_hard_tasks}
\begin{adjustbox}{max width=\linewidth}
\begin{tabular}{llrrrrr}
\toprule
Model & Pairs & N & $r_S$ & 95\% interval & $r_{\mathrm{within}}$ & 95\% interval \\
\midrule
Assignment & Both positive & 58 & 0.100 & [-0.188, 0.351] & -0.020 & [-0.385, 0.327] \\
Row-order & All finished & 63 & 0.390 & [0.176, 0.592] & 0.275 & [-0.025, 0.540] \\
Sums + ranks & All finished & 64 & 0.393 & [0.134, 0.595] & 0.330 & [-0.014, 0.591] \\
\bottomrule
\end{tabular}
\end{adjustbox}\end{table}

\begin{table}[H]\centering\small
\caption{Cross-formulation ranks on jointly finished harder-sample calls. MILP pairs use LP iterations. Assignment versus CP-SAT uses LP iterations versus branches, changing both solver and counter.\vspace{2.0mm}}\label{tab:app-review_hard_ranks}
\begin{adjustbox}{max width=\linewidth}
\begin{tabular}{lllrr}
\toprule
Task & Left & Right & Pairs & $r_S$ \\
\midrule
certification & Assignment & Raw natural & 64 & -0.022 \\
certification & Assignment & Sums + ranks & 64 & -0.068 \\
certification & Raw natural & Sums + ranks & 64 & 0.677 \\
certification & Assignment & CP-SAT & 64 & 0.409 \\
completion & Assignment & Raw natural & 64 & 0.144 \\
completion & Assignment & Sums + ranks & 64 & 0.145 \\
completion & Raw natural & Sums + ranks & 64 & 0.360 \\
completion & Assignment & CP-SAT & 64 & 0.084 \\
\bottomrule
\end{tabular}
\end{adjustbox}\end{table}

\begin{table}[H]\centering\small
\caption{Paired rank-strengthening effects on harder-sample LP work. All calls finish with positive counts. Less/equal/more compare the strengthened model with raw natural; ratios are formed per puzzle before taking their median.\vspace{2.0mm}}\label{tab:app-submission_hard_strengthening}
\begin{adjustbox}{max width=\linewidth}
\begin{tabular}{lrrrrrr}
\toprule
Task & N & Less & Equal & More & Median rank/raw & $r_S$ \\
\midrule
completion & 64 & 50 & 0 & 14 & 0.630 & 0.360 \\
certification & 64 & 59 & 0 & 5 & 0.611 & 0.677 \\
\bottomrule
\end{tabular}
\end{adjustbox}\end{table}

The separate Norvig check contains \ExternalN\ retained puzzles. Assignment and CP-SAT finish both tasks on all of them, with positive LP and branch counts respectively on the same two puzzles. Assignment completion processes one node in each of those cases, whereas certification processes more than one. Raw natural finishes \ExternalNaturalFindDone\ completion and \ExternalNaturalCertDone\ certification calls. The extensive ties in this collection motivate the harder sample; they do not establish a common scale across counters.

\subsection{Strengthening on the diagnostic selection}\label{app:task-comparisons}
Table~\ref{tab:variants} gives marginal coverage and positive counts. Table~\ref{tab:pairedcoverage} makes paired coverage explicit before comparing effort. Every row partitions the same ninety diagnostic puzzles. Only the jointly finished calls contribute to the rank and less/equal/more columns; unfinished effort is not imputed. Assignment variants finish all calls at the zero-LP floor. The order-exclusion variants have unchanged completion matrices.

\begin{table}[H]\centering\footnotesize
\caption{Paired diagnostic coverage and LP effort relative to raw natural. Both, Raw only, Variant only and Neither describe completion status under the same cap. The remaining columns use only Both: less/equal/more compare variant iterations with raw iterations.\vspace{2.0mm}}\label{tab:pairedcoverage}
\begin{adjustbox}{max width=\linewidth}\begin{tabular}{llrrrrrrrr}
\toprule Task & Variant & Both & Raw only & Variant only & Neither & $r_S$ & Less & Equal & More \\ \midrule
completion & Natural + singles & 88 & 0 & 2 & 0 & 0.279 & 24 & 64 & 0 \\
completion & All-order & 88 & 0 & 0 & 2 & 1.000 & 0 & 88 & 0 \\
completion & Row-order & 88 & 0 & 0 & 2 & 1.000 & 0 & 88 & 0 \\
completion & Unit sums & 88 & 0 & 1 & 1 & 0.783 & 20 & 64 & 4 \\
completion & Sums + ranks & 87 & 1 & 1 & 1 & 0.554 & 23 & 64 & 0 \\
certification & Natural + singles & 86 & 0 & 4 & 0 & 0.244 & 22 & 64 & 0 \\
certification & All-order & 85 & 1 & 2 & 2 & 0.952 & 8 & 63 & 14 \\
certification & Row-order & 85 & 1 & 2 & 2 & 0.954 & 12 & 63 & 10 \\
certification & Unit sums & 85 & 1 & 3 & 1 & 0.740 & 17 & 64 & 4 \\
certification & Sums + ranks & 85 & 1 & 3 & 1 & 0.580 & 20 & 64 & 1 \\
\bottomrule\end{tabular}\end{adjustbox}\end{table}

For the opening hard case, \texttt{SUD-17-0213}, separate presolve inspections leave nonempty models for raw natural, row-order and unit-sum certification, but detect infeasibility for sums plus ranks and for common singles. For \texttt{SUD-17-0008}, the raw exclusion is already detected infeasible in presolve. These observations locate where each diagnostic operation stops; they do not show that HiGHS executes the explicit singles sequence.

\subsection{Matched presolve sensitivity}\label{app:presolve}
The design reuses the fixed five-per-count benchmark sample (seed 20260926) and all \PreExternalN\ harder records, selected before these measurements. Independent integer enumeration checks every selected puzzle. All fresh calls receive original givens and no completion hint, retained search state or external singles closure. They differ only in the task, formulation and presolve setting being compared. The initial cap is \PreTimeLimit\ seconds, with four single-threaded workers, solver seed zero and job order shuffled with seed 20261023. Solver tolerances and gaps match the primary MILP runs.

\begin{table}[H]\centering\small
\caption{Presolve sensitivity at the common ten-second cap. F denotes completion and C certification. Each row uses the same puzzles under both settings. Unfinished calls are retained as time-limited observations, not assigned completed effort.\vspace{2.0mm}}\label{tab:presolve-coverage}
\begin{tabular}{llrrr}
\toprule
Population and model & Task & Puzzles & Finished on & Finished off \\
\midrule
Benchmark: assignment & F & 320 & 320 & 320 \\
Benchmark: assignment & C & 320 & 320 & 320 \\
Benchmark: raw natural & F & 320 & 320 & 319 \\
Benchmark: raw natural & C & 320 & 318 & 318 \\
Harder: assignment & F & 64 & 64 & 64 \\
Harder: assignment & C & 64 & 64 & 64 \\
Harder: raw natural & F & 64 & 64 & 64 \\
Harder: raw natural & C & 64 & 64 & 59 \\
\bottomrule
\end{tabular}

\end{table}

The \PrePrimaryUnresolved\ unfinished calls involve \PreExtendedN\ puzzles. Every condition on those puzzles is rerun at \PreExtendedLimit\ seconds, including conditions already finished. All \PreExtendedCalls\ extension calls finish. Each affected puzzle then uses the extended records for all eight conditions; other puzzles retain their ten-second records. Thus the cap is common within each puzzle, not throughout the collection. All \PreRepeatN\ previously finished calls repeat with identical LP and node counts. The fixed-limit coverage and extended paired comparisons remain distinct.

\begin{table}[H]\centering\small\setlength{\tabcolsep}{4pt}
\caption{Paired LP-iteration comparisons after the matched cap extension. Zero columns count finished zero-iteration calls. ``Less'' and ``More'' count puzzles with fewer or more iterations when presolve is off; unlisted pairs tie. The final column is the Spearman correlation between the on and off counts on all $n$ pairs. A dash means the on counter is constant, so correlation is undefined.\vspace{2.0mm}}\label{tab:presolve-pairs}
\begin{tabular}{llrrrrrr}
\toprule
Population and model & Task & $n$ & Zero on & Zero off & Less & More & $r_s$ \\
\midrule
Benchmark: assignment & F & 320 & 320 & 313 & 0 & 7 & --- \\
Benchmark: assignment & C & 320 & 320 & 311 & 0 & 9 & --- \\
Benchmark: raw natural & F & 320 & 301 & 238 & 11 & 71 & 0.532 \\
Benchmark: raw natural & C & 320 & 301 & 238 & 11 & 71 & 0.534 \\
Harder: assignment & F & 64 & 6 & 0 & 19 & 45 & 0.086 \\
Harder: assignment & C & 64 & 0 & 0 & 31 & 33 & 0.458 \\
Harder: raw natural & F & 64 & 0 & 0 & 22 & 42 & 0.096 \\
Harder: raw natural & C & 64 & 0 & 0 & 18 & 46 & 0.461 \\
\bottomrule
\end{tabular}

\end{table}

\begin{table}[H]\centering\small
\caption{Completion--certification LP-iteration association after the matched extension. The first block retains shared zeros. The final two rows use exactly the same benchmark puzzles with positive counts in both tasks under both settings. Percentile intervals use five hundred seeded paired resamples, within clue count for the benchmark and across puzzles for the harder sample. A dash denotes undefined correlation. These intervals condition on the selected records and the completed extension.\vspace{2.0mm}}\label{tab:presolve-tasks}
\begin{tabular}{llrrr}
\toprule
Population and model & Presolve & $n$ & Task $r_s$ & Bootstrap interval \\
\midrule
Benchmark: assignment & on & 320 & --- & --- \\
Benchmark: assignment & off & 320 & 0.8831 & $[0.7083, 1.0000]$ \\
Benchmark: raw natural & on & 320 & 0.9997 & $[0.9994, 1.0000]$ \\
Benchmark: raw natural & off & 320 & 0.9994 & $[0.9988, 0.9997]$ \\
Harder: assignment & on & 64 & 0.0307 & $[-0.2388, 0.2773]$ \\
Harder: assignment & off & 64 & 0.0750 & $[-0.1843, 0.3436]$ \\
Harder: raw natural & on & 64 & 0.4038 & $[0.2004, 0.5859]$ \\
Harder: raw natural & off & 64 & 0.2180 & $[-0.0659, 0.4635]$ \\
\midrule
\multicolumn{5}{l}{Same benchmark puzzles positive in all four raw-natural calls} \\
Benchmark: raw natural & on & 19 & 0.7982 & $[0.5499, 0.9587]$ \\
Benchmark: raw natural & off & 19 & 0.8333 & $[0.7160, 0.9187]$ \\
\bottomrule
\end{tabular}

\end{table}

Switching off main-model presolve leaves in-search deductions available. Some logs still record presolve within sub-MIP heuristics. Thus this experiment changes one solver option, rather than removing every simplifying operation. The opening examples were checked separately under the same HiGHS version with a sixty-second cap. Table~\ref{tab:opening-presolve} shows that their certification ordering persists with presolve off, while their completion ordering reverses. All completions match the independently verified reference; all exclusion models are infeasible.

\begin{table}[H]\centering\small
\caption{The two Figure~\ref{fig:opening} puzzles under both presolve settings. F denotes completion and C certification. LP columns count reported iterations, not LP solves. Both original puzzles are feasible and unique, and the infeasible models are their exclusions.\vspace{2.0mm}}\label{tab:opening-presolve}
\begin{tabular}{llrrrr}
\toprule
Puzzle & Task & LP on & LP off & Nodes on & Nodes off \\
\midrule
\texttt{SUD-17-0008} & C & 0 & 16,681 & 0 & 212 \\
\texttt{SUD-17-0213} & C & 16,367 & 26,353 & 327 & 513 \\
\texttt{SUD-17-0008} & F & 0 & 20,452 & 0 & 251 \\
\texttt{SUD-17-0213} & F & 8,404 & 15,841 & 38 & 177 \\
\bottomrule
\end{tabular}

\end{table}

\subsection{Digit relabelling}\label{app:relabelling}
Table~\ref{tab:app-revision_permutations} supplies the coverage and variance denominators for Figure~\ref{fig:permutations}. Each base has its identity labelling and twenty distinct seeded permutations. Complete-group restriction avoids assigning final effort to unfinished calls, but does not make these selected bases representative.

\begin{table}[H]\centering\footnotesize
\caption{Relabelling coverage and log-effort variation. Variances use complete labelling groups only. Variable-base counts use finished values without imputing censored outcomes.\vspace{2.0mm}}\label{tab:app-revision_permutations}
\begin{adjustbox}{max width=\linewidth}
\begin{tabular}{llrrrrr}
\toprule
Variant & Task & Done / N & Complete bases & Variable bases & Within var. & Base-mean var. \\
\midrule
Assignment & certification & 252/252 & 12 & 0 & 0.000 & 0.000 \\
Assignment & completion & 252/252 & 12 & 0 & 0.000 & 0.000 \\
Raw natural & certification & 236/252 & 10 & 6 & 5.839 & 13.035 \\
Raw natural & completion & 247/252 & 10 & 6 & 5.446 & 11.785 \\
Row-order & certification & 239/252 & 10 & 6 & 5.828 & 12.899 \\
Row-order & completion & 247/252 & 10 & 6 & 5.446 & 11.785 \\
\bottomrule
\end{tabular}
\end{adjustbox}\end{table}

\subsection{Supporting predictive results}\label{app:predictive-results}
Table~\ref{tab:clipping} retains the principal natural-model scores and the nonnegative-prediction check. Figure~\ref{fig:prediction} shows the separate grid and digit ablations and the incremental blocks. Table~\ref{tab:app-revision_cv_cp} compares candidate, residual-indicator and residual-count predictors for CP work. The residual-only population has 2,333\ eligible records after the minimum stratum-size rule. These are held-out associations within the stated populations, not estimates of transfer to the harder collection.

\begin{table}[H]\centering\footnotesize
\caption{Natural-model prediction with unchanged folds and fits. Neg. counts negative predictions; Minimum is the smallest prediction. Subscript $+$ denotes clipping at zero. Full is $M_3$; Count is $M_0$.\vspace{2.0mm}}\label{tab:clipping}
\begin{adjustbox}{max width=\linewidth}
\begin{tabular}{lllrrrrrr}
\toprule
Task & Population & Block & Neg. & Minimum & $R^2$ & $R^2_+$ & MAE & MAE$_+$ \\
\midrule
completion & All & Count & 0 & 0.000 & 0.653 & 0.653 & 0.416 & 0.416 \\
completion & All & Full & 15,501 & -0.605 & 0.660 & 0.661 & 0.481 & 0.445 \\
completion & Positive & Count & 0 & 7.951 & 0.266 & 0.266 & 0.848 & 0.848 \\
completion & Positive & Full & 0 & 7.490 & 0.335 & 0.335 & 0.809 & 0.809 \\
certification & All & Count & 0 & 0.000 & 0.654 & 0.654 & 0.436 & 0.436 \\
certification & All & Full & 15,614 & -0.594 & 0.660 & 0.661 & 0.502 & 0.466 \\
certification & Positive & Count & 0 & 9.157 & 0.283 & 0.283 & 0.837 & 0.837 \\
certification & Positive & Full & 0 & 8.085 & 0.357 & 0.357 & 0.792 & 0.792 \\
\bottomrule
\end{tabular}
\end{adjustbox}\end{table}

\begin{table}[H]\centering\footnotesize
\caption{CP deterministic-work scores for $\log(1+10^6d)$. Every model includes count intercepts. The indicator and residual count enter separately after the candidate block. Positive residual means $U>0$.\vspace{2.0mm}}\label{tab:app-revision_cv_cp}
\begin{adjustbox}{max width=\linewidth}
\begin{tabular}{lllrrr}
\toprule
Task & Population & Predictors & N & $R^2$ & MAE \\
\midrule
completion & All & Count & 32,000 & 0.338 & 0.390 \\
completion & All & + Candidates & 32,000 & 0.340 & 0.404 \\
completion & All & + Residual indicator & 32,000 & 0.898 & 0.115 \\
completion & All & + Residual count & 32,000 & 0.947 & 0.083 \\
completion & Positive residual & Count & 2,333 & 0.049 & 1.078 \\
completion & Positive residual & + Residual count & 2,333 & 0.495 & 0.764 \\
certification & All & Count & 32,000 & 0.298 & 0.390 \\
certification & All & + Candidates & 32,000 & 0.300 & 0.415 \\
certification & All & + Residual indicator & 32,000 & 0.891 & 0.097 \\
certification & All & + Residual count & 32,000 & 0.924 & 0.072 \\
certification & Positive residual & Count & 2,333 & 0.029 & 1.106 \\
certification & Positive residual & + Residual count & 2,333 & 0.335 & 0.839 \\
\bottomrule
\end{tabular}
\end{adjustbox}\end{table}

\begin{table}[H]\centering\footnotesize
\caption{Occurrence of positive natural-model LP work. All calls are resolved after recovery. These scores assess occurrence, not the size of a positive count.\vspace{2.0mm}}\label{tab:app-revision_occurrence}
\begin{adjustbox}{max width=\linewidth}
\begin{tabular}{llrrr}
\toprule
Task & Predictors & AUC & Log loss & Brier \\
\midrule
completion & Count & 0.986 & 0.070 & 0.0215 \\
completion & + Residual count & 0.990 & 0.060 & 0.0188 \\
certification & Count & 0.987 & 0.069 & 0.0212 \\
certification & + Residual count & 0.990 & 0.059 & 0.0186 \\
\bottomrule
\end{tabular}
\end{adjustbox}\end{table}

\begin{table}[H]\centering\footnotesize
\caption{Lower generation-range sensitivity (seventeen to twenty-three clues). Count denotes $M_0$ and the residual-count row $M_3$. Separate fits do not identify a causal generator effect.\vspace{2.0mm}}\label{tab:app-revision_cv_low}
\begin{adjustbox}{max width=\linewidth}
\begin{tabular}{lllrrr}
\toprule
Task & Population & Predictors & N & $R^2$ & MAE \\
\midrule
completion & Low-count finished & Count & 3,500 & 0.352 & 3.426 \\
completion & Low-count finished & + Residual count & 3,500 & 0.443 & 3.142 \\
certification & Low-count finished & Count & 3,500 & 0.358 & 3.600 \\
certification & Low-count finished & + Residual count & 3,500 & 0.446 & 3.303 \\
\bottomrule
\end{tabular}
\end{adjustbox}\end{table}

\subsection{Exploratory clue-deletion margins}\label{app:margin-results}
The fixed \NMargin-puzzle selection has exact margin optima. Margin optimisation omits hard clue fixings, excludes the reference and minimises violated original clues, with a redundant lower bound of one based on input uniqueness. The independent CP check omits that bound. Each witness is a valid alternative grid; exact values require matching solver bounds.

All \SubmissionMarginPositiveN\ positive natural-certification LP outcomes have margin one. Table~\ref{tab:app-review_margin_counts} shows that variation in margin and positive LP work do not overlap within exact counts. For the exploratory regression, $M_4=M_3+\beta_\rho\rho(P)$ uses identical eligible records and folds, with unpenalised count intercepts. Table~\ref{tab:app-revision_margin_gain} gives its paired gain over $M_3$. The CP comparison is a different outcome and also gives no improvement under this specification.

\begin{table}[H]\centering\footnotesize
\caption{Margin sample by clue range. Positive LP counts use recovered natural certification. Varying counts enumerates exact-count strata with more than one observed margin.\vspace{2.0mm}}\label{tab:app-review_margin_counts}
\begin{adjustbox}{max width=\linewidth}
\begin{tabular}{lrrrrrrr}
\toprule
Clues & N & $\rho=1$ & 2 & 3 & 4 & Positive LP & Varying counts \\
\midrule
17--23 & 35 & 35 & 0 & 0 & 0 & 18 & 0 \\
24--30 & 35 & 35 & 0 & 0 & 0 & 1 & 0 \\
31--40 & 50 & 50 & 0 & 0 & 0 & 0 & 0 \\
41--60 & 100 & 86 & 14 & 0 & 0 & 0 & 10 \\
61--80 & 100 & 8 & 45 & 39 & 8 & 0 & 18 \\
\bottomrule
\end{tabular}
\end{adjustbox}\end{table}

\begin{table}[H]\centering\footnotesize
\caption{Incremental margin gain, $M_4$ versus $M_3$, on the same held-out records. Positive values favour the margin. Intervals condition on fitted folds. Natural uses $\log(1+\mathrm{LP})$ and CP uses $\log(1+10^6d)$.\vspace{2.0mm}}\label{tab:app-revision_margin_gain}
\begin{adjustbox}{max width=\linewidth}
\begin{tabular}{lrrrr}
\toprule
Outcome & $\Delta R^2$ & 95\% interval & MAE gain & 95\% interval \\
\midrule
natural & $-3.85\times10^{-7}$ & [$-5.68\times10^{-5}$, $5.49\times10^{-5}$] & -0.0017 & [-0.0027, -0.0007] \\
cp & $-4.65\times10^{-5}$ & [$-9.33\times10^{-5}$, $-2.06\times10^{-5}$] & -0.0005 & [-0.0008, -0.0002] \\
\bottomrule
\end{tabular}
\end{adjustbox}\end{table}

\subsection{Matched clue additions and overlapping generators}\label{app:margin-chains}
Every sequence starts from an independently verified unique base and adds only correct clues. Its integer completion set remains the same singleton. Table~\ref{tab:app-revision_chains} gives adjacent-pair denominators and increase sizes. Assignment and CP finish all 2,520 adjacent pairs per task. Raw-natural comparisons use only the finished pairs shown; all remaining adjacent pairs are unresolved. These denominators count pairs, not individual calls. Intervals in Table~\ref{tab:app-revision_chain_intervals} resample bases, retaining both addition orders.

\begin{table}[H]\centering\footnotesize
\caption{Nested additions: finished pairs, direction of change, affected bases and median increases. MILP counts are LP iterations; CP uses deterministic work with tie tolerance $10^{-12}$. Up/non-tied conditions on a counter change; all-pair increase rates are in Table~\ref{tab:app-revision_chain_intervals}.\vspace{2.0mm}}\label{tab:app-revision_chains}
\begin{adjustbox}{max width=\linewidth}\begin{tabular}{llrrrrrrr}
\toprule Model & Task & Done & Up & Down & Tied & Up/non-tied & Bases & Median up \\ \midrule
CP-SAT & certification & 2,520 & 27 & 68 & 2,425 & 0.284 & 9 & $2.97\times10^{-6}$ \\
CP-SAT & completion & 2,520 & 357 & 657 & 1,506 & 0.352 & 20 & $4.00\times10^{-8}$ \\
Assignment & certification & 2,520 & 0 & 0 & 2,520 & -- & 0 & -- \\
Assignment & completion & 2,520 & 0 & 0 & 2,520 & -- & 0 & -- \\
Raw natural & certification & 2,463 & 27 & 104 & 2,332 & 0.206 & 14 & 4,664 \\
Raw natural & completion & 2,489 & 38 & 123 & 2,328 & 0.236 & 17 & 3,860 \\
\bottomrule\end{tabular}\end{adjustbox}\end{table}

\begin{table}[H]\centering\footnotesize
\caption{Base-cluster bootstrap intervals for nonconstant empirical increase rates. Both orders stay together. Constant outcomes are omitted rather than interpreted as guarantees.\vspace{2.0mm}}\label{tab:app-revision_chain_intervals}
\begin{adjustbox}{max width=\linewidth}
\begin{tabular}{llrrr}
\toprule
Model & Task & Up fraction & CI lower & CI upper \\
\midrule
CP-SAT & certification & 0.011 & 0.006 & 0.017 \\
CP-SAT & completion & 0.142 & 0.133 & 0.150 \\
Raw natural & certification & 0.011 & 0.007 & 0.015 \\
Raw natural & completion & 0.015 & 0.011 & 0.020 \\
\bottomrule
\end{tabular}
\end{adjustbox}\end{table}

The same chains extend augmentation into the benchmark deletion generator's range. Table~\ref{tab:app-revision_overlap} retains the overlap at twenty-four through thirty clues. Its forty augmentation records at each count share twenty bases and are dependent. This small, matched construction check helps interpret the twenty-three-to-twenty-four-clue boundary, but does not estimate a population-wide generator effect.

\begin{table}[H]\centering\footnotesize
\caption{Overlapping construction ranges: augmentation chains and benchmark deletion instances. The augmentation records share bases and are not independent draws.\vspace{2.0mm}}\label{tab:app-revision_overlap}
\begin{adjustbox}{max width=\linewidth}
\begin{tabular}{rlrrrr}
\toprule
Clues & Construction & N & Singles fraction & Mean residual & Median residual \\
\midrule
24 & augmentation & 40 & 0.925 & 2.40 & 0.0 \\
24 & deletion & 500 & 0.508 & 19.64 & 0.0 \\
25 & augmentation & 40 & 0.950 & 1.75 & 0.0 \\
25 & deletion & 500 & 0.520 & 18.63 & 0.0 \\
26 & augmentation & 40 & 0.950 & 1.75 & 0.0 \\
26 & deletion & 500 & 0.528 & 17.63 & 0.0 \\
27 & augmentation & 40 & 0.975 & 0.65 & 0.0 \\
27 & deletion & 500 & 0.664 & 12.38 & 0.0 \\
28 & augmentation & 40 & 1.000 & 0.00 & 0.0 \\
28 & deletion & 500 & 0.764 & 8.23 & 0.0 \\
29 & augmentation & 40 & 1.000 & 0.00 & 0.0 \\
29 & deletion & 500 & 0.798 & 6.99 & 0.0 \\
30 & augmentation & 40 & 1.000 & 0.00 & 0.0 \\
30 & deletion & 500 & 0.852 & 4.62 & 0.0 \\
\bottomrule
\end{tabular}
\end{adjustbox}\end{table}

\end{document}